\documentclass{article}
\usepackage{amsmath,amsthm,amsfonts,graphicx}
\usepackage{multirow}
\usepackage{slashbox}
\usepackage{multirow}
\def\smallddots{\mathinner{\raise7pt\hbox{.}\raise4pt\hbox{.}\raise1pt\hbox{.}}}
\def\smallsdots{\mathinner{\raise1pt\hbox{.}\raise4pt\hbox{.}\raise7pt\hbox{.}}}

\DeclareMathOperator{\diag}{diag}

\DeclareMathOperator{\rank}{rank}

\numberwithin{equation}{section}
\numberwithin{table}{section}
\newtheorem{theorem}{Theorem}[section]
\newtheorem{lemma}{Lemma}[section]

\newtheorem{corollary}{Corollary}[section]
\newtheorem{fact}{Fact}[section]

\newtheorem{example}{Example}[section]
\newtheorem{definition}{Definition}[section]

\newtheorem{fig.}{FIGURE}[section]

\begin{document}
\title{\bf Fast Approximate 
Computations with 
Cauchy Matrices 
and Polynomials
\thanks  {Some results of this paper have been presented at 
ILAS'2013, Providence, RI, 2013,
at CASC'2013, Berlin, Germany, 2013,
and at
CSR'2014, Moscow, Russia, 2014.}
}

\author{Victor Y. Pan  \\
Departments of Mathematics and Computer Science \\
Lehman College and the Graduate Center of the City University of New York \\
Bronx, NY 10468 USA \\
victor.pan@lehman.cuny.edu \\
http://comet.lehman.cuny.edu/vpan/  \\
} 

 \date{}


\maketitle


\begin{abstract}


Multipoint polynomial evaluation  and interpolation
are fundamental for modern symbolic and numerical computing.
The known algorithms solve both problems  over any field
of constants in nearly linear  arithmetic time,
but the cost grows to quadratic for numerical solution. 
We fix  this discrepancy: our new
 numerical algorithms run in
 nearly linear arithmetic time. At first we restate our goals as the
multiplication of an $n\times n$ Van\-der\-monde matrix
 by a vector and the solution of a Van\-der\-monde linear system of $n$ equations.
Then we  transform the matrix into a  
 Cauchy structured matrix with some special features. By exploiting them,
we approximate the  matrix by a generalized hierarchically semiseparable matrix,
which is a structured matrix of a different class.
Finally we accelerate our solution to the original problems  by 
 applying Fast Multipole Method to the latter matrix. 
Our resulting numerical algorithms run in nearly optimal arithmetic time  
when they perform the above fundamental computations with polynomials,
Van\-der\-monde matrices,
 transposed Van\-der\-monde matrices,
and a large class of Cauchy and Cauchy-like matrices. 
Some of our techniques
may be of independent interest.
\end{abstract}






\section{Introduction}\label{s1}


\subsection{The background and our progress}\label{sbcgr}


Multipoint  polynomial evaluation and interpolation are  
 fundamental for modern  symbolic and numerical computing.
 The known FFT-based algorithms run in nearly linear arithmetic time, but need quadratic time 
if the precision of computing is  restricted, e.g., to the IEEE standard double precision 
 (cf. \cite{BF00}, \cite{BEGO08}).
Our algorithms solve the problems in nearly linear 
arithmetic time even under such a restriction.

At first we restate the original tasks 
as the problems of multiplication of a 
Van\-der\-monde 
matrix
by a vector and the solution of a nonsingular
Van\-der\-monde linear system of equations, then 
transform the input   
matrix into a
matrix with the  structure of Cauchy type,   and 
finally
apply the
numerically stable FMM to a  generalized HSS matrix
that approximates the latter matrix.\footnote{``HSS" and ``FMM" are  the acronyms for ``Hierarchically Semiseparable"
and ``Fast Multipole Method".}
``Historically HSS representation
is just a special case of the representations commonly exploited in the
FMM literature" \cite{CDG06}. 
We refer the reader to
the books
  \cite{B10},
\cite{VVM},  \cite{EGH13},
and the bibliography therein
 for the  
FMM and the
 HSS matrices.

Our resulting fast algorithms apply to the following computational problems:

\begin{itemize}
\item
 multipoint  polynomial evaluation  and interpolation,
\item
multiplication by a vector  
of  
a Van\-der\-monde matrix,
its transpose,  
and, more generally,
matrices with the structures of Cauchy or  Van\-der\-monde  type,
\item
the solution of a linear system of equations with 
these coefficient matrices,
\item
rational interpolation 
and multipoint evaluation
associated with Cauchy matrix computations.
\end{itemize}

Some of our techniques
can be of independent interest
(cf. their extension in  
  \cite{Pa}).  

As in the papers \cite{MRT05},
 \cite{CGS07},  \cite{XXG12}, and
\cite{XXCB14}, we  count arithmetic operations
in the field
$\mathbb C$ of complex numbers with no rounding errors, but
 our algorithms are essentially reduced to
application of the  celebrated  algorithms 
of FFT and FMM,
having  stable numerical performance.


\subsection{Related works and our techniques}\label{srlt}


{\em Our progress can be viewed  as a new demonstration 
of the power of combining  the
transformation of matrix structures 
of \cite{P90}
with the
 FMM/HSS techniques.}

The paper  \cite{P90} has proposed some efficient techniques 
for the transformation of the four 
most popular matrix structures of Toeplitz, Hankel,
Cauchy, and Van\-der\-monde types into each other
and then showed that these techniques enable us to
readily extend any efficient algorithm
for the inversion of a matrix having one of these structures   
to efficient inversion of the matrices having 
structures of the three other types.
The  papers \cite{PSLT93} and \cite{PZHY97}
have extended these techniques to the acceleration of 
multipoint polynomial evaluation, but have not invoked
the  FMM and achieved only
 limited progress.  
 Short Section 9.2 of {\cite{PRT92} has pointed out 
some potential benefits of combining FMM
  with the algorithm of the paper \cite{G88},
but has not developed  that idea. 
The papers
\cite{P95} and \cite{DGR96} applied FMM and some other advanced techniques
in order to accelerate
approximate polynomial evaluation at a set of real points.

The closest neighbors of our present study
are the papers
 \cite{MRT05},
 \cite{CGS07},  \cite{XXG12}, \cite{XXCB14},  and \cite{P15}. The
former four papers
approximate the solution of
Toeplitz, Hankel, Toep\-litz-like, and Han\-kel-like linear systems
of equations in nearly linear arithmetic time,
 versus the 
  cubic time of the  classical numerical algorithms and 
the previous record quadratic time of \cite{GKO95}.
All five
papers \cite{GKO95},  
\cite{MRT05},
 \cite{CGS07},  \cite{XXG12}, and \cite{XXCB14} 
begin with the transformation of an input
 matrix 
into
a Cauchy-like one,
by specializing the cited technique
of \cite{P90}. 
 Then \cite{GKO95} 
continued by
exploiting the invariance of the Cauchy  
structure in 
row 
interchange, while the other four
papers apply the
numerically stable FMM in order
to operate efficiently with HSS approximations
of the basic
Cauchy matrix.

We incorporate
the powerful  FMM/HSS techniques,
but extend them nontrivially.
The papers \cite{GKO95},  
\cite{MRT05},
 \cite{CGS07},  \cite{XXG12}, and \cite{XXCB14}  
handle just the special 
Cauchy matrix $C=(\frac{1}{s_i-t_j})_{i,j=0}^{m-1,n-1}$ for which $m=n$,
$\{s_0,\dots,s_{n-1}\}$ is the set of the $n$-th
roots of unity and $\{t_0,\dots,t_{n-1}\}$ is
the set of the other $2n$-th roots of unity.
Our fast Van\-der\-monde multipliers and 
 solvers bring us to a
 subclass of Cauchy matrices
$C=(\frac{1}{s_i-t_j})_{i,j=0}^{m-1,n-1}$
rather than to a single matrix:
we still
assume that the 
knots $t_0,\dots,t_{n-1}$
are equally spaced on the unit circle, but 
impose no restriction on the knots $s_0,\dots,s_{m-1}$
and arrive at the matrices

\begin{equation}\label{eqcvm}
C_{{\bf s},f}=\Big (\frac{1}{s_i-f\omega^j}\Big )_{i,j=0}^{m-1,n-1},
\end{equation}
for any complex numbers $f,s_0,\dots,s_{m-1}$
and $\omega=\exp(2\pi\sqrt {-1}/n)$
denoting a primitive $n$th root of unity.

We call the matrices $C_{{\bf s},f}$ {\em CV matrices},
 link them to Van\-der\-monde 
matrices,
and devise efficient approximation algorithms
 that multiply a CV matrix by a vector, solve a 
nonsingular CV linear system of equations, and hence perform
multipoint polynomial evaluation and interpolation.
In order to achieve this progress, we 
work with
 extended HSS matrices,
associated with  CV matrices via
 a  proper partition of the complex plane: 
 we 
bound the numerical rank of the off-block-tri\-dia\-gonal 
blocks (rather than 
the off-block-dia\-gonal blocks, as is customary) and allow distinct rectangular  blocks to share
 row indices. Extension of the FMM/HSS techniques to such matrix classes 
was not straightforward and required additional care.

The paper \cite{P15} revisited the method of the 
transformation of matrix structures (traced back to  \cite{P90}), 
 recalled its techniques in some details, extended them,  and finally
outlined our present approach to polynomial  interpolation and multipoint 
evaluation
in order to demonstrate the power of that method once again. 
The paper included only one half of a page to 
  HSS matrices and about as much to
the reduction of the polynomial evaluation and interpolation
to computations with CV matrices.
No room has been left for  
the description of nontrivial computations  with generalized HSS matrices 
(having cyclic block tridiagonal part),
to which the original problems are reduced.
Furthermore  the 
competing fast algorithms for polynomial and rational 
interpolation and multipoint evaluation
of \cite{MB72}, \cite{H72}, and \cite{GGS87}
have not been cited.

 We fill this void by describing in some detail
the omitted algorithms for generalized HSS computation,
by linking polynomial and rational interpolation and multipoint 
 evaluation to CV matrices, 
by demonstrating the inherent numerical instability of the 
 algorithms of \cite{MB72}, \cite{H72}, and \cite{GGS87},
and by presenting some numerical tests, in particular for comparison of
  numerical stability of our algorithms with that 
of \cite{MB72}.
Also we more fully and more clearly cover the approximation 
of CV matrices by generalized HSS matrices.


\subsection{Organization of our paper}\label{sorg}

   
In the next section we recall some
 basic results for  matrix computations. In Section \ref{s3}
 we recall the problems of
polynomial and rational 
 evaluation and interpolation and
represent them
in terms of Van\-der\-monde,
Cauchy,  and CV matrices.
Sections 
 \ref{s2} and  \ref{s3}  (on the Background) make up Part I of our paper.

Sections 
  \ref{sqs1} and \ref{sextbd}  (on the Extended HSS Matrices) make up Part II,
where at first  we  recall the known algorithms for 
fundamental computations with
HSS matrices 
and then
extend the algorithms
to generalized HSS  matrices having cyclic block tridiagonal part.
Part II can be read independently of Section  \ref{s3}.

Sections 
 \ref{snrqs} and  \ref{scmplcv}  (on Computations with the CV Matrices and Extensions) 
make up Part 
III of the paper.
In Section \ref{snrqs} we approximate
a CV matrix by generalized HSS matrices
and estimate the complexity of the 
resulting numerical 
computations with CV matrices.
In Section \ref{scmplcv}
we  comment on the extensions and implementation
of our algorithms,
in particular  the extension to computations with 
Van\-der\-monde
 matrices and polynomials.
The results of Section \ref{snrqs}
imply our main results
because we have already 
 reduced  polynomial 
interpolation and  multipoint evaluation to computations with CV matrices
in Part I and have elaborated upon fast computations with generalized HSS matrices in Part II.
 
Part III 
uses Section \ref{s32} and equations (\ref{eqvs1}) 
 and (\ref{eqvs-1}) of Part I
(which support the cited reduction 
to CV matrices)
and 
 Theorem \ref{thalphbt} and
Corollary \ref{coext1} of Part II
(where we estimate the cost of computations with generalized HSS matrices),
but otherwise can be read independently of Parts I  and II.

Sections \ref{ststs} and \ref{scnc}
make up Part IV of the paper.
In  Section \ref{ststs} we report the results of our numerical tests.
In Section \ref{scnc}
 we briefly summarize our study.


\bigskip
\bigskip

{\huge PART I: BACKGROUND}


\section{Definitions and auxiliary results}\label{s2}


\subsection{Some basic definitions for matrix computations}\label{sbdef}


$O=O_{m,n}$ is the $m\times n$ matrix filled with zeros.
$I=I_n$ is the $n\times n$ identity matrix.

$M^T$ is the transpose of a matrix $M$, 
$M^H$ is its Hermitian transpose. 

$\diag(B_0,\dots,B_{k-1})=\diag(B_j)_{j=0}^{k-1}$
is a $k\times k$ block diagonal matrix 
with diagonal blocks $B_0,\dots,B_{k-1}$.

Both $(B_0~\dots~B_{k-1})$ and 
$(B_0~|~\dots~|~B_{k-1})$
denote a $1\times k$ block matrix with
$k$ blocks $B_0,\dots,B_{k-1}$.


$||M||=||M||_2$ denotes the spectral norm of a matrix $M$.

For an $m\times n$ 
matrix $M=(m_{i,j})_{i,j=0}^{m-1,n-1}$,
 write $|M|=\max_{i,j}|m_{i,j}|$, and so $||M||\le \sqrt{mn}~|M|$,
but for a set $\mathcal S$ we write $|\mathcal S|$ to denote its cardinality.

An $m\times n$  matrix $U$ is  {\em unitary} 
 if $U^HU=I_n$ or $UU^H=I_m$, and then
$||U||=1$.

 ``$\ll$"  stands for ``much less" quantified in context.


\subsection{Submatrices, rank, and generators}\label{ssrg}


An $m\times n$ matrix $M$ has a nonunique {\em generating pair 
$(F,G^T)$ of
a length} $\rho$
if $M=FG^T$ for two matrices $F$ of size $m\times \rho$
and $G$ of size $n\times \rho$.
 The  minimum
length of a generating pair of a matrix is equal to its rank.

$\mathcal R(B)$ and $\mathcal C(B)$ are the  index sets
of the rows and columns of its submatrix $B$, respectively.
For two sets $\mathcal I\subseteq \{1,\dots,m\}$ and $\mathcal J\subseteq \{1,\dots,n\}$,
define the submatrix
 $B=M(\mathcal I,\mathcal J)=(m_{i,j})_{i\in \mathcal I,j\in \mathcal J}$
such that 
  $\mathcal R(B)=\mathcal I$ and  $\mathcal C(B)=\mathcal J$.
Write
$M(\mathcal I,.)=M(\mathcal I,\mathcal J)$
 if $\mathcal J=\{1,\dots,n\}$.
Write $M(.,\mathcal J)=M(\mathcal I,\mathcal J)$
 if $\mathcal I=\{1,\dots,m\}$.

\begin{theorem}\label{thgnrt}  
A matrix $M$ has rank at least $\rho$
if and only if it has a 
nonsingular $\rho\times \rho$ submatrix $M(\mathcal I,\mathcal J)$.
 
If $\rank (M)=\rho$, then
$M=M(.,\mathcal I)M(\mathcal I,\mathcal J)^{-1}M(\mathcal J,.)$.
\end{theorem}
The theorem defines 
two generating pairs
 $(M(.,\mathcal I),M(\mathcal I,\mathcal J)^{-1}M(\mathcal J,.)$
 and 
$(M(.,\mathcal I)M(\mathcal I,\mathcal J)^{-1},M(\mathcal J,.)$ and
a {\em generating triple}
$(M(.,\mathcal I),M(\mathcal I,\mathcal J)^{-1},M(\mathcal J,.))$ 
 of a length $\rho$
for a matrix  $M$. 
We call
such pairs and triples 
{\em generators}.
One can obtain some generators of the minimum length for a given matrix
by computing its SVD $U\Sigma V$ or its less costly rank revealing factorizations
such as ULV and URV factorizations in \cite{CGS07}, \cite{XXG12}, 
and \cite{XXCB14}, where 
the  matrices U and V 
are unitary, $\Sigma $ is diagonal, and L and R 
are triangular (cf. \cite[Section 5.6.8]{GL13}).
For efficient  alternative techniques, 
some of which use randomization or heuristics,
see
\cite{GOS08},
\cite{GT01}, 
\cite{HMT11},
\cite{LWMRT},
\cite{M11},
\cite{M11a},
\cite{PQY15},
\cite{T00},
\cite{W14},
\cite{XXG12},
and the references therein.


\subsection{Small-norm approximation and 
perturbation}\label{seps}

Hereafter we deal with perturbations within a positive tolerance 
$\xi$. 
(One may think of machine epsilon, but in this paper we just assume that
$\xi$ is small in context.)

A matrix $\tilde M$ is a $\xi$-{\em approximation} of
a matrix $M$ if $||\tilde M-M||\le \xi||M||$.

A $\xi$-generator of a matrix $M$
is a generator of its $\xi$-approximation.

The $\xi$-{\em rank} of a matrix $M$ is
the integer
$\min_{||\tilde M-M||\le \xi||M||} \rank (\tilde M)$.

A matrix $M$ is {\em ill-con\-di\-tioned} if its rank exceeds its
numerical rank. 


\section{Polynomial and rational
evaluation and interpolation as
operations with
structured  matrices}\label{s3}


\subsection{Four classes of
structured  matrices. Cauchy and Van\-der\-monde matrices}\label{s31}


Recall the four classes of highly  popular 
structured matrices, that is,
{\em Toeplitz} matrices
 $T=\left(t_{i-j}\right)_{i,j=0}^{m-1,n-1}$,
{\em Hankel} matrices
 $H=\left(h_{i+j}\right)_{i,j=0}^{m-1,n-1}$,
{\em Van\-der\-monde} matrices
$V=V_{\bf s}=(s_i^{j})_{i,j=0}^{m-1,n-1}$,  and 
{\em Cauchy} matrices 
$C=C_{{\bf s},{\bf t}}=\Big(\frac{1}{s_i-t_j}\Big)_{i,j=0}^{m-1,n-1}$. 
(Some authors call the transpose $V^T$ a Van\-der\-monde matrix.)
The $mn$ entries of such a structured $m\times n$
 matrix are
defined by at most $m+n$
 parameters.

These classes have been extended to the four 
more general classes of  matrices having structures of 
Toeplitz, Hankel, Van\-der\-monde, and Cauchy types.
Each such an $m\times n$ matrix is naturally defined by its
{\em displacement generator} $FH$ 
where  $F$ and  $G$ are $m\times d$ 
and $d\times n$ matrices, respectively, and  where 
$d\ll \min\{m,n\}$, that is,  $\min\{m,n\}$
 exceeds greatly the integer $d$    
(cf. \cite{P01},  \cite{P15}).

We mostly work with Van\-der\-monde and Cauchy matrices
and next recall some of their basic properties.

The scalars $s_0,\dots,s_{m-1},t_0,\dots,t_{n-1}$ 
define the Van\-der\-monde and 
Cauchy matrices $V_{\bf s}$ and  $C_{\bf s,t}$, and 
we call them {\em knots}.   
If we shift 
the knots of a Cauchy matrix or
 scale them by a constant, we arrive at a Cauchy matrix
again:
$aC_{a{\bf s},a{\bf t}}=C_{\bf s,t}$ for $a\neq 0$ and
$C_{{\bf s}+a{\bf e},{\bf t}+a{\bf e}}=C_{\bf s,t}$ for
${\bf e}=(1,\dots,1)^T$.


\begin{theorem}\label{thcvdet}
(i) An $m\times n$ Van\-der\-monde matrix $V_{\bf s}=(s_i^{j})_{i,j=0}^{m-1,n-1}$
 has full rank if and only if all $m$ knots  $s_0,\dots,s_{m-1}$
are distinct. 
(ii)  An $m\times n$ Cauchy matrix $C_{\bf {s,t}}=
\Big(\frac{1}{s_i-t_j}\Big)_{i,j=0}^{m-1,n-1}$ 
is well defined and has full rank if and only if all its  $m+n$ knots
are distinct.
\end{theorem}

The four cited matrix structures have quite distinct features.
In particular 
the 
matrix
structure of Cauchy type
is invariant in 
row and column interchange,
in contrast to the structures of
Toeplitz and Hankel types. 
This structure is stable in 
shifting and scaling its basic knots 
unlike the structure of   Van\-der\-monde type. 

The paper \cite{P90}, however,
has transformed the matrices of any of the four classes
into the matrices of the three other classes simply by means of multiplication by 
Hankel, Van\-der\-monde, and transposed or inverse Van\-der\-monde matrices. 
Then the paper has showed that
 such transforms
{\em readily extend any efficient matrix inversion
algorithm for matrices of one of the four classes to the matrices of 
 the three other classes,} and similarly for the computation of determinants and the solution of linear systems of equations.  

Presently we apply a simple specialization of this general technique 
for devising efficient 
approximation algorithms for  Van\-der\-monde matrix computations
 linked to polynomial evaluation and interpolation.


\subsection{Four computational problems}\label{s32}


{\bf Problem 1.
 Multipoint 
Polynomial evaluation or Van\-der\-monde-by-vec\-tor multiplication.}


\noindent INPUT: $m+n$ complex scalars $p_0,\dots,p_{n-1};
s_0,\dots,s_{m-1}$.


\noindent OUTPUT: $m$ complex scalars $v_0,\dots,v_{m-1}$ satisfying
$v_i=p(s_i)$ for $p(x)=p_0+p_1x+\cdots+p_{n-1}x^{n-1}$ and $i=0,\dots,m-1$
or equivalently
$V{\bf p}={\bf v}$ for $V=V_{\bf s}=(s_i^{j})_{i,j=0}^{m-1,n-1}$,
${\bf p}=(p_j)_{j=0}^{n-1}$, and ${\bf v}=(v_i)_{i=0}^{m-1}$.



{\bf Problem 2. Polynomial interpolation or the solution of a Van\-der\-monde linear system.}


\noindent INPUT: $2n$ complex scalars $v_0,\dots,v_{n-1};s_0,\dots,s_{n-1}$,
 the last $n$ of them distinct.


\noindent OUTPUT: $n$ complex scalars $p_0,\dots,p_{n-1}$ satisfying the above
equations 
for $m=n$.

\medskip

{\bf Problem 3. Multipoint rational evaluation or Cau\-chy-by-vec\-tor multiplication.}

\noindent INPUT: $2m+n$ complex scalars $s_0,\dots,s_{m-1};
t_0,\dots,t_{n-1};v_0,\dots,v_{m-1}$.

\noindent OUTPUT: $m$ complex scalars $v_0,\dots,v_{m-1}$ satisfying
$v_i=\sum_{j=0}^{n-1}\frac{u_j}{s_i-t_j}$ for $i=0,\dots,m-1$
or equivalently
$C{\bf u}={\bf v}$ for $C=C_{{\bf s},{\bf t}}=
\Big(\frac{1}{s_i-t_j}\Big)_{i,j=0}^{m-1,n-1},~
{\bf u}=(u_j)_{j=0}^{n-1}$, and ${\bf v}=(v_i)_{i=0}^{m-1}$.


{\bf Problem 4. Rational interpolation or the solution of a Cauchy linear system of equations.}

\noindent INPUT: $3n$ complex scalars 
$s_0,\dots,s_{n-1};t_0,\dots,t_{n-1};v_0,\dots,v_{n-1}$,
the first $2n$ of them distinct.

\noindent OUTPUT: $n$ complex scalars $u_0,\dots,u_{n-1}$ satisfying 
the above equations 
for $m=n$.


\subsection{The arithmetic  complexity of Problems 1--4}\label{s33}


The  algorithm
of  \cite{MB72}
solves  Problem 1 by using $O((m+n)\log^2(n)\log(\log (n)))$ arithmetic  operations.
This complexity bound  has been extended
to the solution of 
Problems 2 in \cite{H72}, 3 in \cite{GGS87}, and 4
(see equation (\ref{eqfhr}) below)
and
is within a  factor of $\log (n)\log(\log (n))$ from the optimum
 \cite{BM75}.

The cited algorithms supporting this bound
require extended precision of computing 
and fail already for the input polynomials of moderate degree if
the precision is restricted to the IEEE standard double precision
(cf. Table \ref{SuperfastEval}). 
The approach relies heavily on  computing with extended precision.
Already the fast 
polynomial division algorithm 
requires computations with high precision
for the worst case input, and the problem
is aggravated  in the recursive 
 fan-in processes of polynomial multiplication and division 
in the algorithms of \cite{MB72}, \cite{H72}, and  \cite{GGS87}.
Moreover, the following argument demonstrates that we  must add  at least
$n$  bits of precision when these algorithms
compute the Lagrange auxiliary polynomial 
with the roots $s_0,\dots,s_{n-1}$.

\medskip

{\bf Problem 5. Computation of the polynomial coefficients from its roots.}

\noindent INPUT: $n$ complex scalars $s_0,\dots,s_{n-1}$.

\noindent OUTPUT: the coefficients of
the  polynomial
$l(x)=\prod_{i=0}^{n-1}(x-s_i)$.

\medskip

In order to observe the need for 
the precision increase, notice that the constant coefficient 
has absolute value $\prod_{j=0}^{n-1}|s_i|$,
which turns into $2^n$ if, say, $s_i=2$  for all $i$, but
the coefficient of $x^{\lfloor n/2\rfloor}$ has the order of
$2^n$ even if $s_i=1$  for all $i$.  
The restriction of using bounded (e.g., double)
precision of computing rules out using the cited fast algorithms,
and the known double precision algorithms for Problems 1--4
require quadratic arithmetic time 
 (cf. 
 \cite{BF00},
 \cite{BEGO08}). 

This pessimistic outcome, however, 
does not apply to
the important special case  where the  knots $s_i$
 are the $n$th roots of 1, that is, where
$s_i=\omega^{i}$ for
 $\omega=\omega_n=\exp(2\pi \sqrt{-1}/n)$,
 $i=0,\dots,n-1$.
In this case, 
$V_{\bf s}=(\omega^{ij})_{i,j=0}^{m-1,n-1}$ and
 Problems 1 (for $m=n$) and 2 turn into the computation of
the forward and inverse 
{\em discrete Fourier transforms}, respectively.
Hereafter we use the acronyms {\em DFT} and {\em IDFT}
and write 
$\Omega=\frac{1}{\sqrt n}(\omega^{ij})_{i,j=0}^{n-1}$.
Notice that 
 $\Omega=\Omega^T$  and
$\Omega^H=\Omega^{-1}=
\frac{1}{\sqrt n}(\omega^{-ij})_{i,j=0}^{n-1}$
 are unitary matrices.
Based on FFT, one  can perform the DFT and IDFT, that is, 
can solve
 Problems 1 and 2 in this special case,
 by using  
 bounded precision of computing
and involving only $O(n\log (n))$ arithmetic operations  \cite[Problem 2.4.2]{P01}.
  

\subsection{Cauchy--Van\-der\-monde links and their impact on Problems 1 and 2}\label{s34}


The following equation, traced to \cite{K68} on \cite[page 110]{P01}, links
 Problems 1 and 2 to
Cauchy matrices, 
\begin{equation}\label{eqfhr}
C_{\bf s,t}=\diag(l(s_i)^{-1})_{i=0}^{m-1}V_{\bf s}V^{-1}_{\bf t}\diag(l'(t_j))_{j=0}^{n-1},~
l(x)=\prod_{j=0}^{n-1}(x-t_j).
\end{equation}

For ${\bf t}=f\cdot(\omega^{j})_{j=0}^{n-1}$, $f\neq 0$, the knots $t_j$ are the scaled $n$th 
roots of 1,
$l(x)=x^n-f^n$, $l'(x)=nx^{n-1}$,
$V_{\bf t}=\sqrt n~\Omega\diag(f^{j})_{j=0}^{n-1}$,
$V_{\bf t}^{-1}=\frac{1}{\sqrt n}\diag(f^{-j})_{j=0}^{n-1}\Omega^H$. 
Likewise for ${\bf s}=e\cdot(\omega^{i})_{i=0}^{n-1}$, $e\neq 0$, the knots $s_i$ are the scaled $n$th 
roots of 1, $V_{\bf s}=\sqrt n~\Omega\diag(e^{i})_{i=0}^{n-1}$ and
 $V_{\bf s}^{-1}=\frac{1}{\sqrt n}\diag(e^{-j})_{j=0}^{n-1}\Omega^H$.

Write $C_{{\bf s},f}=(\frac{1}{s_i-f\omega^j})_{i,j=0}^{m-1,n-1}$ for $f\neq 0$
and  $C_{e,{\bf t}}=(\frac{1}{e\omega^i-t_j})_{i,j=0}^{m-1,n-1}$ for  $e\neq 0$
and obtain from 
(\ref{eqfhr}) that
%
 

\begin{equation}\label{eqvs1}
V_{\bf s}=~
\frac{f^{1-n}}{\sqrt n}\diag\Big(s^n_i-f^n\Big)_{i=0}^{m-1}C_{{\bf s},f}\diag(\omega^{j})_{j=0}^{n-1}\Omega\diag(f^{j})_{j=0}^{n-1},
\end{equation}
\begin{equation}\label{eqvt-1}
V_{\bf t}^{-1}=~
\frac{1}{\sqrt n}~\diag(e^{-i})_{i=0}^{m-1}\Omega^H\diag(l(e^i))_{i=0}^{m-1}C_{e,{\bf t}}\diag \Big (\frac{1}{l'(t_j)}\Big )_{j=0}^{n-1},~{\rm and}
\end{equation}

\begin{equation}\label{eqvs-1}
V_{\bf s}^{-1}=
\sqrt n\diag(f^{-j})_{j=0}^{n-1}\Omega^H\diag(\omega^{-j})_{j=0}^{n-1}C_{{\bf s},f}^{-1}\diag\Big(\frac{f^{n-1}}{s^n_i-f^n}\Big)_{i=0}^{n-1}~{\rm for}~m=n.
\end{equation}
These  expressions
link 
Van\-der\-monde matrices and their inverses
to the $m\times n$ CV matrices  
$C_{{\bf s},f}$
of equation (\ref{eqcvm}) and
the  $n\times m$ {\em CV}$^T$ {\em matrices} 
$C_{e,{\bf t}}=-C^T_{{\bf t},e}=\Big(\frac{1}{e\omega^i-t_j}\Big)_{i,j=0}^{n-1,m-1}$ (for $e\neq 0$), that is, Cauchy matrices with 
an arbitrary  knot set $\mathcal T=\{t_0,\dots,t_{n-1}\}$
and with the knot set $\mathcal S=\{s_i=e\omega^{i},~i=0,\dots,m-1\}$.
More details on the subjects of this section can be found in \cite{Pb}.


\medskip

\medskip
\medskip
\medskip
\medskip

{\huge PART II: EXTENDED HSS MATRICES}


\section{Quasi\-separable and HSS matrices}\label{sqs1}



\subsection{Quasi\-separable  matrices
and generators}\label{shss}


\begin{definition}\label{defhss0}
Suppose that an $m\times n$ matrix $M$
is represented as a  $k\times k$ block matrix
with 
a block diagonal 
 $\widehat \Sigma=(\Sigma_0,\dots,\Sigma_{k-1})$. 
 Let
 $\chi(\widehat \Sigma)$ denote 
the overall number of the entries of
all its $k$ diagonal blocks $\Sigma_0,\dots,\Sigma_{k-1}$
and let
$\chi(\widehat \Sigma)\ll mn$, that is,
let $mn$  greatly  exceed
$\chi(\widehat \Sigma)$.
Furthermore let $l$ and  $u$ denote the maximum ranks of the sub-
and superdiagonal
 blocks of the matrix $M$, respectively. 
Then the matrix $M$ is 
 $(l,u)$-{\em quasi\-separable}.
By replacing ranks with $\xi$-ranks we
 define a $(\xi,l,u)$-{\em quasi\-separable}
matrix.
\end{definition}

The definition generalizes the class of banded matrices and their inverses:
 a matrix having a lower  bandwidth $l$
and an upper bandwidth $u$ as well as its inverse
(if defined) are $(l,u)$-quasi\-separable.

In order to operate with 
$(l,u)$-quasi\-separable  matrices
 efficiently, one exploits their  
representation 
 with {\em quasi\-separable generators}, 
demonstrated  
by the following
$4\times 4$ 
example and 
defined below in general form,


\begin{equation}\label{eqqus}
M=\begin{pmatrix}
\Sigma_0  &  S_0T_1   &   S_0B_1T_2~  &   S_0B_1B_2 T_3  \\
P_1Q_0  &  \Sigma_1   &   S_1T_2  &     S_1B_2T_3   \\
P_2A_1Q_0  &   P_2Q_1   &   \Sigma_2   &    S_2T_3   \\
P_3A_2A_1Q_0  & ~P_3A_2Q_1   &  P_3 Q_2   &  \Sigma_3
\end{pmatrix}.
\end{equation}

By generalizing this example we arrive at the following definition.

\begin{definition}\label{defhssgen} (Cf. Table \ref{tabgen}.)
Suppose that an $m\times n$ matrix $M$
is represented as a  $k\times k$ block matrix with 
a block diagonal 
 $\widehat \Sigma=(\Sigma_0,\dots,\Sigma_{k-1})$ such that 
$\chi(\widehat \Sigma)\ll mn$.
(We reuse these assumptions of Definition \ref{defhss0}.)

Furthermore suppose that 
a set $\{\mathcal I_1,\dots,\mathcal I_k\}$ partitions
the set $\{1,\dots,m\}$; 
a set $\{\mathcal J_1,\dots,\mathcal J_k\}$ partitions
the set $\{1,\dots,n\}$, and
there exists
 a six-tuple 
$\{P_i$, $Q_h$, $S_h$, $T_i$, $A_g$, $B_g\}$ such that
$M(\mathcal I_i,\mathcal J_h)=P_iA_{i-1}\cdots A_{h+1}Q_h$ and 
$M(\mathcal I_h,\mathcal J_i)=S_hB_{h+1}\cdots B_{i-1}T_i$ 
for $0\le h<i<k$.    

Here $P_i$, $Q_h$, and $A_g$ are $|\mathcal I_i|\times l_i$,
$l_{h+1}\times |\mathcal J_h|$, and $l_{g+1}\times l_g$  matrices,
respectively, and

$S_h$, $T_i$ and $B_g$ are $|\mathcal I_h|\times u_{h+1}$,
$u_{i}\times |J_i|$, and $u_g\times u_{g+1}$
  matrices,
respectively, \\
for $g=1,\dots,k-2$, $h=0,\dots,k-2$, $i=1,\dots,k-1$.

Then the six-tuple $\{P_i$, $Q_h$, $S_h$, $T_i$, $A_g$, $B_g\}$ 
is an $(l,u)$-{\em quasi-separable generator} 
of the matrix $M$,
and the integers 
$l=\max_g\{l_g\}$
and $u=\max_h\{u_h\}$
are the {\em lower and upper lengths}
or {\em orders}
of this generator.
\end{definition} 


\begin{table}[ht]
  \caption{The sizes of quasi\-separable generators of Definition
\ref{defhssgen}}
  \label{tabgen}
  \begin{center}
    \begin{tabular}{| c | c | c | c | c | c | c |}
      \hline
      $P_i$ & $Q_h$ & $A_g$ & $S_h$ & $T_i$ & $B_g$ \\ \hline
 $|\mathcal I_i|\times l_i$ & $l_{h+1}\times |\mathcal J_h|$ & $l_{g+1}\times l_g$ & 
$|\mathcal I_h|\times u_{h+1}$ & $u_{i}\times |J_i|$ & $u_g\times u_{g+1}$ \\ \hline
    \end{tabular}
  \end{center}
\end{table}



 \begin{theorem}\label{thhss1}
(Cf. \cite{B10}, \cite{VVM},  
\cite{X12},  \cite{EGH13}, and the bibliography therein.)
A matrix $M$ is $(l,u)$-quasi-separable
if and only if it has a (nonunique) representation 
via $(l,u)$-quasi-separable generators.
 \end{theorem} 

By virtue of this theorem one can redefine the $(l,u)$-quasi\-separable matrices as 
those 
 representable 
with the  
families
of quasi\-separable generators
 $\{P_h$, $Q_i$,  $A_g\}$ and $\{S_h$, $T_i$, $B_g\}$
that have
lower and upper orders
  $l$ and $u$,
respectively.
 Definitions \ref{defhss0}
and  \ref{defhssgen} provide
 two  useful insights 
into the properties of  
these matrices.
The third equivalent 
definition in Section \ref{snbrm} (cf. Theorem \ref{thhssqs}) 
provides yet another insight and is linked to the study
of the Cauchy matrix $C_{1,\omega_{2n}}$ in 
\cite{CGS07},  \cite{XXG12},  \cite{XXCB14}. 
Various definitions, equivalent or closely related to those above,
have been introduced by a number of authors 
(cf. \cite{VVM}, \cite{B10}, \cite{EGH13}, and the references therein). 
In particular the related study of 
 $H$-matrices and $H^2$-matrices in \cite{H99},
\cite{T00}, \cite{BH02}, \cite{GH03}, \cite{B09}, \cite{B10},  
and references therein was the basis for the
software libraries  HLib, www.hlib.org, and H2Lib,
http://www.h2lib.org/, 
https://github.com/H2Lib/H2Lib, developed at the
Max Planck Institute for Mathematics in the Sciences.
 

\subsection{Operations with 
quasi\-separable matrices: definitions and demonstration}\label{sqsdfdm}


Next we cover some basic operations with 
matrices represented with 
$(l,u)$-quasi\-separable generators.




\begin{definition}\label{defalphabeta}
Given 
 diagonal blocks $\Sigma_q$, 
$q=0,\dots,k-1$, 
of an  $(l,u)$-quasi\-separable  matrix $M$
 and  $(l,u)$-quasi\-separable generators for all its 
sub- and super-diagonal blocks, 
let $\alpha (M)$ and $\beta(M)$ denote the
arithmetic cost of 
 computing 
the vectors
 $M{\bf u}$  
and  
$M^{-1}{\bf u}$, respectively,
maximized over all normalized vectors ${\bf u}$, $|{\bf u}|=1$,
and minimized over all algorithms.
Write  $\beta(M)=\infty$
if the matrix  $M$ is singular.
 $\alpha_{\xi}(M)$  and $\beta_{\xi}(M)$
replace  the bounds 
 $\alpha(M)$  and $\beta(M)$, respectively,
provided that instead of the evaluation of the vectors $M{\bf u}$
and
$M^{-1}{\bf u}$, respectively, 
we approximate them within the error bounds
 $\xi||M{\bf u}||$ and
 $\xi||M^{-1}{\bf u}||$, respectively. 
\end{definition}

The straightforward algorithm supports the following bound.


\begin{theorem}\label{thmbv}
$\alpha (M)\le 2(m+n)\rho-\rho-m$ where a generating pair 
of length $\rho$ defines an $m\times n$ matrix $M$.
\end{theorem}

The following estimates
for computations with  quasiseparable matrices   
 extend the well-known estimates in the case of banded matrices.

\begin{theorem}\label{thhssd}  \cite{DV98}, \cite{H99}, \cite{EG02}.
Suppose that an  $(l,u)$-quasi\-separable  matrix $M$ 
of size $m\times n$ is defined by 
its $m_q\times n_q$  
 diagonal blocks $\Sigma_q$, 
$q=0,\dots,k-1$, 
such that $\sum_{q=0}^{k-1}m_q=m$, $\sum_{q=0}^{k-1}n_q=n$, and 
$s=\sum_{q=0}^{k-1}m_qn_q=O( (l+u) (m+n))$
and by the generators of length at most $l$ and at most $u$ for its 
sub- and superdiagonal blocks, respectively.  

\noindent (i) Then 
$\alpha(M)\le 
2\sum_{q=0}^{k-1}((m_q+n_q)(l+u)+s)+2l^2k+2u^2k=O((l+u)(m+n))$ and

\noindent (ii) $\beta(M)=O(\sum_{q=0}^{k-1}((l+u)^2(l+u+n_q)n_q+n_q^3))$
 if $m_q=n_q$ for all $q$ and if
 the  matrix $M$ is nonsingular.
\end{theorem}


\begin{example}\label{ex2} (Cf. Figures 2 and 3.)
Let us multiply by a vector $\bf v$ the matrix $M$ of equation (\ref{eqqus}).  \\
(i) At first view it as $2\times 2$ block matrix with diagonal blocks 
$\bar \Sigma_1=\begin{pmatrix}
\Sigma_0  &  S_0T_1     \\
P_1Q_0  &  \Sigma_1   
\end{pmatrix}$
and $\bar \Sigma_2=\begin{pmatrix}
\Sigma_2   &    S_2T_3   \\
  P_3 Q_2   &  \Sigma_3
\end{pmatrix}$; multiply  
the 
blocks 
$\begin{pmatrix}
 S_0B_1T_2~  &   S_0B_1B_2 T_3  \\
   S_1T_2  &     S_1B_2T_3   
\end{pmatrix}$
and 
$\begin{pmatrix}
P_2A_1Q_0  &   P_2Q_1      \\
P_3A_2A_1Q_0  & ~P_3A_2Q_1 
\end{pmatrix}$
by two subvectors of the vector ${\bf v}$. \\
(ii) Then multiply the blocks $S_0T_1$,
$P_1Q_0$, $S_2T_3$, and $P_3 Q_2$  of the matrices $\bar \Sigma_1$
and $\bar \Sigma_2$ of smaller sizes  by four subvectors of the vector ${\bf v}$.

Perform the computations at both stages fast if the given generators of the blocks 
have small length. \\
(iii) Then multiply the four diagonal blocks
$\Sigma_1$, $\Sigma_2$, $\Sigma_3$, and $\Sigma_4$
by four subvectors of the vectors ${\bf v}$.
Perform these computations fast because the four blocks have a small overall number of entries. \\
(iv) Finally obtain the vector $M{\bf v}$ by properly summing the products. 
\end{example}
  

\subsection{Fast multiplication with 
recursive merging of diagonal blocks: outline}\label{snbrmout}


In Example \ref{ex2}  we multiply the matrix $M$ by a vector 
by using generators for only 6 out of its 22  sub- and super-diagonal blocks.
Next we extend the above demonstration to 
multiplication of a general quasi\-separable matrix $M$
by a vector 
by using a small fraction of all generators.

\begin{definition}\label{defqs}  
Suppose that  $M=(M_0~ |~\dots~|~ M_{k-1})$ is a
 $1\times k$ block matrix with $k$ {\em  block 
columns} $M_q$,
each partitioned into a
{\em diagonal block} $\Sigma_q$
and a {\em 
 neutered block column} $N_q$, $q=0,\dots,k-1$
 (cf. our Figures 1--3 and \cite[Section 1]{MRT05}).
 Such a matrix
is  {\em  $\rho$-neutered}
if its every  neutered block column $N$ is represented as 
$N=FH$ or  $N=FSH$ where $F$ of size $h\times r$, $S$  of size $r\times r$,
and $H$ of size $r\times k$ are its {\em generator
matrices} and $r\le \rho$. Call such a 
pair or triple a {\em length $r$ generator of the 
neutered block} $N$ and call $r$ its {\em length}. 
A $\xi$-approximation of such a matrix is called 
{\em  $(\xi,\rho)$-neutered}.
\end{definition}

In {\bf Figure 1} 
the diagonal blocks are black
and  
the neutered block columns are 
gray or white.

 \medskip

{\bf FIGURE 1}

\begin{figure}[ht!]\label{figure2}
\centering
\includegraphics[width=90mm]{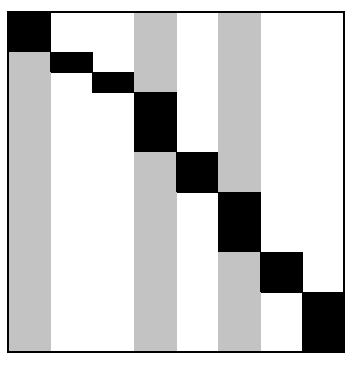}
\caption{ FIGURE 2}
\end{figure}

\clearpage


In {\bf Figure 2} the diagonal blocks from Figure 1
(marked by 
black color) are merged pairwise
 into their diagonal unions, each made up of four
blocks. Two of them (from Figure 1) are marked by
black color, and the two other by 
gray
color. 
The new neutered block columns are either white or
 gray, but their gray color is lighter.
The new (larger) diagonal blocks of Figure 2 are merged  
 pairwise into the diagonal blocks of Figure 3, each 
made up of two black and two gray blocks, and its two neutered block
columns are white.


 \medskip

 {\bf FIGURE  2}

\medskip

\begin{figure}[ht!]\label{figure3}
\centering
\includegraphics[width=90mm]{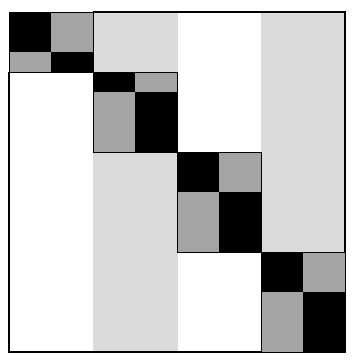}
\end{figure}


\clearpage

 {\bf FIGURE  3}

\begin{figure}[ht!]\label{figure4}
\centering
\includegraphics[width=90mm]{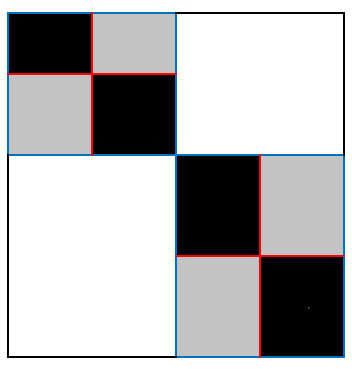}
\end{figure}

 \clearpage




\begin{theorem}\label{thntr01} 
Suppose that an $m\times n$ matrix $M$ is a 
 $\rho$-neutered $k\times k$ 
 block matrix and that we are given $k$ generators
of length at most  
$\rho$ for all its $k$
 neutered block columns as well as all the 
$\chi(\widehat \Sigma)$ entries in 
 the $k$ diagonal blocks $\Sigma_0,\dots,\Sigma_{k-1}$.
Then
$$\alpha(M)\le 2\chi(\widehat \Sigma)+(2m+2n-1)k\rho=O(\chi(\widehat \Sigma)+(m+n)k\rho).$$
\end{theorem}

\begin{proof}
Multiply the diagonal blocks by vectors in the straightforward way
and multiply the  neutered block columns by vectors 
by using the representation with generators.

Formally write $M=M'+\diag(\Sigma_q)_{q=0}^{k-1}$.
Notice that $\alpha(M)\le 2\chi(\widehat \Sigma)+\alpha(M')+m$.
The  neutered block columns of the matrix $M$ 
share their entries with the matrix $M'$,
whose other entries are zeros.
So the $k$ pairs 
$(F_0,G_0),\dots,(F_{k-1},G_{k-1})$
together form a single generating pair
of a length at most $k\rho$
for the matrix $M'$.
Therefore 
$\alpha(M')\le (2m+2n-1)k\rho-m$
by virtue of Theorem \ref{thmbv}.
\end{proof} 


The upper bound on $\alpha(M)$ of Theorem \ref{thntr01} is 
sufficiently small unless the integers $k$
or $\chi(\widehat \Sigma)$ are large. Unfortunately
we cannot bound both of these integers
at once, but we can circumvent the problem by applying 
the algorithm of Theorem \ref{thntr01} recursively.
 We begin with a partition
of the matrix $M$ defined by a few  diagonal blocks 
that are   $\rho$-neutered matrices  themselves. Then we 
multiply  neutered block columns fast (by using their generators),
partition 
the  diagonal blocks into smaller diagonal blocks 
and  neutered block columns, and apply the same techniques recursively
until we decrease the overall  number of entries of the remaining diagonal
blocks below  a fixed tolerance bound of order $m+n$ or $(m+n)\rho$.

We can begin with $k=2$ and $\chi(\widehat \Sigma)\approx 0.5 n^2$
and then double the integer $k$ and roughly halve the integer $\chi(\widehat \Sigma)$ 
in every recursive step.
Then overall we deal with only $O(m+n)$  neutered block columns
and their generators and therefore multiply the matrix $M$ by a vector 
by using $O((m+n)\rho)$
 arithmetic operations in all these
 recursive steps, thus matching the  cost bounds in part (i) 
of Theorem \ref{thhssd}.


\subsection{HSS and balanced HSS matrices and
the cost of basic
operations with them}\label{snbrm}


Let us supply formal definitions and formal derivation of the latter estimates
by applying the recursive process in the opposite direction,
where at first  the integer $k$ is large and then is  recursively
doubled, while the diagonal blocks are small at first
and then are merged recursively pairwise.


\begin{definition}\label{defunion}
Fix two positive integers $l$ and $q$ such that $l+q\le k$ and then
{\em merge} the $l$  block columns  \\
$M_{q},M_{q+1},\dots,M_{q+l-1}$, the $l$
diagonal blocks $\Sigma_{q},\Sigma_{q+1},\dots,\Sigma_{q+l-1}$,
and the $l$  neutered block columns
$N_{q},N_{q+1}$,
$\dots,N_{q+l-1}$
 into their {\em union}
$M_{q,l}=M(.,\cup_{j=0}^{l-1}\mathcal C(\Sigma_{q+j}))$,
 their {\em diagonal union}
$\Sigma_{q,l}$,
 and
their {\em neutered union}
$N_{q,l}$, respectively,
 such that
$\mathcal R(\Sigma_{q,l})=
\cup_{j=0}^{l-1} \mathcal R(\Sigma_{q+j})$ and
every block column $M_{q,l}$
is partitioned into the diagonal union
$\Sigma_{q,l}$ and
the neutered union
$N_{q,l}$.
\end{definition}

Define  {\em recursive merging}
of all diagonal blocks
$\Sigma_0,\dots,\Sigma_{k-1}$
by a binary tree whose leaves
are  associated with these blocks
and whose every internal vertex 
is the union of its two children
(see Figure 4).
For every vertex $v$
define the sets $L(v)$ and $R(v)$
of its left and right descendants,
respectively.
If $0\le |L(v)|-|R(v)|\le 1$
for all vertices $v$, then the
binary tree is {\em balanced} and   
 identifies {\em balanced merging}
of its leaves, in our case the diagonal blocks.
We can uniquely
define a balanced tree with $n$ leaves
 by removing  the $2^{l(n)}-n$ rightmost 
 leaves 
of the complete binary tree
that has $2^{l(n)}$ leaves for $l(n)=\lceil\log_2(n)\rceil$.
All leaves of the resulting {\em heap structure}
with $n$ leaves
lie in its two lowest levels.

 \medskip

{\bf FIGURE 4:} Balanced merging of 
 diagonal blocks.

\begin{center}
\setlength{\unitlength}{0.5in}
\begin{picture}(8,4)
\put(4,3.5){\line(-2,-1){2}}
\put(4,3.5){\line(2,-1){2}}
\put(2,2.5){\line(-1,-1){1}}
\put(2,2.5){\line(1,-1){1}}
\put(6,2.5){\line(-1,-1){1}}
\put(6,2.5){\line(1,-1){1}}
\put(1,1.5){\line(-1,-2){0.5}}
\put(1,1.5){\line(1,-2){0.5}}
\put(3,1.5){\line(-1,-2){0.5}}
\put(3,1.5){\line(1,-2){0.5}}
\put(5,1.5){\line(-1,-2){0.5}}
\put(5,1.5){\line(1,-2){0.5}}
\put(7,1.5){\line(-1,-2){0.5}}
\put(7,1.5){\line(1,-2){0.5}}
\put(4,3.5){\makebox(0,0.5){$\Sigma_{0,1,2,3,4,5,6,7}$}}
\put(1.9,2.5){\makebox(0,0.5){$\Sigma_{0,1,2,3}$}}
\put(6.2,2.5){\makebox(0,0.5){$\Sigma_{4,5,6,7}$}}
\put(0.9,1.5){\makebox(0,0.5){$\Sigma_{0,1}$}}
\put(3.2,1.5){\makebox(0,0.5){$\Sigma_{2,3}$}}
\put(4.9,1.5){\makebox(0,0.5){$\Sigma_{4,5}$}}
\put(7.2,1.5){\makebox(0,0.5){$\Sigma_{6,7}$}}
\put(0,0){\makebox(1,0.5){$\Sigma_0$}}
\put(1,0){\makebox(1,0.5){$\Sigma_1$}}
\put(2,0){\makebox(1,0.5){$\Sigma_2$}}
\put(3,0){\makebox(1,0.5){$\Sigma_3$}}
\put(4,0){\makebox(1,0.5){$\Sigma_4$}}
\put(5,0){\makebox(1,0.5){$\Sigma_5$}}
\put(6,0){\makebox(1,0.5){$\Sigma_6$}}
\put(7,0){\makebox(1,0.5){$\Sigma_7$}}
\end{picture}
\end{center}


\begin{definition}\label{defhss}
(i) A block  matrix 
is a {\em balanced $\rho$-HSS}
 matrix if it is 
 $\rho$-neutered throughout 
 the process of  balanced merging
of its diagonal blocks,
that is, if all neutered unions 
of its  neutered block columns involved 
into this process have ranks at most $\rho$.
 This is a $\rho$-{\em HSS} matrix if 
it is 
 $\rho$-neutered throughout 
 any process of recursive merging
of its diagonal blocks.  

(ii) By replacing ranks with $\xi$-ranks we
 define 
{\em balanced}  $(\xi,\rho)$-{\em HSS matrices}
and $(\xi,\rho)$-{\em HSS matrices}.
\end{definition}

 \begin{fact}\label{fahss}
(i) Let  
a matrix be  $\rho_j$-neutered at the $j$-th step of
recursive balanced merging for every $j$. Then this is 
a balanced  $\rho$-HSS matrix
for 
$\rho=\max_j  \rho_j$. 

(ii) Likewise, let  
a matrix be  $(\xi_j,\rho_j)$-neutered at the $j$-th step of
recursive balanced merging for every $j$. Then 
this is a balanced  $(\xi,\rho)$-HSS matrix
for $\xi=\max_j\xi_j$ and $\rho=\max_j  \rho_j$.
\end{fact}

\begin{theorem}\label{thhssqs}
(i) Every $(l,u)$-quasi\-separable matrix $M$
is an $(l+u)$-HSS matrix.

(ii) Every $\rho$-HSS matrix is 
 $(\rho,\rho)$-quasi\-separable.
\end{theorem}

\begin{proof}
A  neutered block column $N_q$
can be partitioned into its  block sub- and super\-diagonal parts
$L_q$ and $U_q$, respectively,
and so 
$\rank(N_q)\le \rank(L_q)+\rank(U_q)$.
This implies that $\rank(N_q)\le l+u$ for $q=0,\dots,k-1$
if the matrix $M$ is $(l,u)$-quasi\-separable, and part (i) is proven.

Next consider the union $N$ of any set of
   neutered block columns of a  matrix $M$. It
 turns into a  neutered block column at some stage 
of appropriate recursive merging. Therefore 
$\rank (N)\le \rho$ where $M$ is a
 $\rho$-HSS  matrix. Now, for every 
off-diagonal block $B$ of a matrix $M$, 
define the set of its  
neutered block columns that share some
column indices with the block $B$
and then notice that the block $B$
is a sub\-matrix 
of the neutered union of this set. Therefore 
$\rank (B)\le \rank (N)\le \rho$, and we obtain part (ii).
\end{proof}

By combining Theorems \ref{thhssd} and \ref{thhssqs}
we obtain the following results.

\begin{corollary}\label{cohssd1}
 Assume a $\rho$-HSS  matrix $M$ 
given
with $m_q\times n_q$  diagonal blocks $\Sigma_q$, $q=0,\dots,k-1$, 
and write $m=\sum_{q=0}^{k-1}m_q$, $n=\sum_{q=0}^{k-1}n_q$, and 
$s=\sum_{q=0}^{k-1}m_qn_q$. 
 Then 

(i)
$\alpha(M)< 2s+4\rho^2k+ 
4\sum_{q=0}^{k-1}(m_q+n_q)\rho=O((m+n)\rho+s)$  
and

 (ii) $\beta(M)=O(\sum_{q=0}^{k-1}((\rho+n_q)\rho^2n_q +n_q^3))$ 
if $m_q=n_q$ for all $q$ and if $\det (M)\neq 0$.
\end{corollary}

For a balanced
$\rho$-HSS matrix $M$
we only have a little weaker representation
than in Theorem \ref{thhss1}, and
so the proof of the estimates of
Corollary \ref{cohssd1}  for $\alpha(M)$
and $\beta(M)$ does not apply, but next we extend these
bounds. Unlike Theorem \ref{thhssd}
and Corollary \ref{cohssd1}, 
 we allow
  $m_q\neq n_q$ for all $q$.

\begin{theorem}\label{thext}
Assume a 
 balanced
$\rho$-HSS 
matrix $M$ 
with
 $m_q\times n_q$  diagonal blocks $\Sigma_q$, $q=0,\dots,k-1$, having
$s=\sum_{q=0}^{k-1} m_qn_q$ entries overall and
write 
$l=\lceil\log_2(k)\rceil$,
$m=\sum_{q=0}^{k-1} m_q$,
$n=\sum_{q=0}^{k-1} n_q$,  $m_+=\max_{q=0}^{k-1} m_q$,
 $n_+=\max_{q=0}^{k-1} n_q$, and
$s\le \min\{m_+n,mn_+\}$. 

(i) Then 
\begin{equation}\label{eqalph}
\alpha(M)< 2s+(m+4(m+n)\rho)l.
\end{equation}
(ii) If $m=n$
 and if the matrix $M$ is nonsingular, then
\begin{equation}\label{eqbet}
\beta(M)=O(n_+s+(n_+^2+\rho n_++l\rho^2)n+(k\rho+n)\rho^2).
\end{equation}
(iii) The same
 bounds (\ref{eqalph}) and (\ref{eqbet})
hold  
for the transpose of a 
 balanced 
$\rho$-HSS matrix $M$
matrix having $n_q\times m_q$ diagonal blocks $\Sigma_q$
for $q=0,\dots,k-1$.    
\end{theorem}

\begin{proof}
Let us readily prove part (i)
by just counting the arithmetic operations involved in
recursive merging.

With no loss of generality assume that 
  the $(l-1)$st (that is, final) stage of a balanced merging process 
has produced a $2\times 2$ block representation 
$$M=\begin{pmatrix}
\bar{\Sigma}_0^{(l)} & \bar S^{(l)}_{01}\bar T^{(l)}_1 &   \\
\bar S^{(l)}_{10}\bar T^{(l)}_0  &  \bar{\Sigma}_1^{(l)}   
\end{pmatrix}$$
where $\bar{\Sigma}_j^{(l)}$ is an 
$\bar m^{(l)}_{j}\times \bar n^{(l)}_{j}$ matrix, $\bar T^{(l)}_j$ is an 
$\bar n^{(l)}_{j}\times \bar {\rho}^{(l)}_{j}$ matrix,
 $\bar {\rho}^{(l)}_{j}\le \rho$, 
$j=0,1$, $\bar m^{(l)}_{1}+\bar m^{(l)}_{2}=m$, and
 $\bar n^{(l)}_{1}+\bar n^{(l)}_{2}=n$.
Clearly $\alpha(M)\le m+
\sum_{j=0}^1\alpha(\bar{\Sigma}_j^{(l)})+
\sum_{j=0}^1\alpha(\bar T^{(l)}_j)+
\alpha(\bar S^{(l)}_{01})+\alpha(\bar S^{(l)}_{10})$.

Apply Theorem \ref{thmbv} and obtain that
$\sum_{j=0}^1\alpha(\bar T^{(l)}_j)+
\alpha(\bar S^{(l)}_{01})+\alpha(\bar S^{(l)}_{10})< 4(m+n)\rho$.

The second last stage of the  balanced merging process
produces a similar $2\times 2$ block representation for each of 
the diagonal blocks $\bar{\Sigma}_j^{(l)}$, $j=0,1$.
Therefore
$\sum_{j=0}^1\alpha(\bar{\Sigma}_j^{(l)})<m+ 4(m+n)\rho + 
\sum_{j=0}^{k(1)}\alpha(\bar{\Sigma}_j^{(1-1)})$
where $\bar{\Sigma}_0^{(1-1)},\dots,\bar{\Sigma}_{k(1)-1}^{(1-1)}$
are the diagonal blocks output at the second
last merging stage (cf. Figures 3 and 4). 

By recursively   
going back through the merging process, obtain
that  
$\alpha(M)<
(m+4(m+n)\rho)l+\sum_{j=0}^{k-1}\alpha(\Sigma_j)$. Here 
$\Sigma_q=\bar{\Sigma}_q^{(0)}$ is an $m_q\times n_q$
matrix for $m_q=\bar m_q^{(0)}$, $n_q=\bar n_{q}^{(0)}$, 
$q=0,\dots,k-1$. 
Hence $\sum_{q=0}^{k-1}\alpha(\Sigma_q)<2\sum_{q=0}^{k-1}m_qn_q=2s$,
implying (\ref{eqalph}).

Part (ii) of the theorem 
 has been supported by the merging and compression algorithm of
\cite{CGS07}. The algorithm has been presented and analyzed in  \cite{CGS07} 
(cf. also \cite{XXG12} and  \cite{XXCB14}) for the
 subclass of balanced $\rho$-HSS matrices, approximating the special matrix
$(\frac{1}{\omega^-f\omega^j})_{i,j=0}^{n-1}$ for $\omega=\exp(2\pi\sqrt{-1}/n)$
and $f=\exp(\pi\sqrt{-1}/n)$, denoting primitive $n$th and $2n$th rooots of 1, respectively,
but both the
algorithm and its analysis are readily extended, and bound (\ref{eqbet}) follows.
All the proofs can be equally applied when rows of the matrix $M$ replace its 
columns and vice versa, and this implies part (iii).
\end{proof}


\begin{corollary}\label{coext}
Under the assumptions of parts (i)--(iii) of
Theorem \ref{thext} suppose 
that  $k\rho=O(n)$
and $n_++\rho=O(\log(n))$. 
Then 
$\alpha(M)=O((m+n)\log^2(n))$ and 
$\beta(M)=O(n\log^3(n))$.
\end{corollary}

For our application to computations with CV matrices we must estimate $\alpha(M)$ and 
$\beta(M)$ for a little more general class of
matrices $M$ defined in the next section.
(Such a matrix has cyclic block tridiagonal part with a 
sufficiently small overall number of entries, say, $O((m+n)\log (m+n))$, 
such that all blocks of the matrix $M$ not overlapping this part
have small rank, say, $O(\log (m+n))$.)  
The algorithms supporting Theorem \ref{thext} 
and Corollary \ref{coext}
are quite readily extended to these matrices
in the next section.
 

\section{Extension from diagonal to tridiagonal blocks}\label{sextbd}



\begin{example}\label{exext}
The following 
matrix 
has eight square or rectangular 
diagonal blocks
$\Sigma_0,\dots,\Sigma_7$ and
becomes
block tridiagonal
if we glue 
its lower and upper boundaries,
\begin{equation}\label{eqbidinv1}
M=\begin{pmatrix}
\Sigma_0 ~ &  B_0~
   &  O~   &  O ~ & O ~ & O ~ & O ~ & A_0  \\
A_1  ~  &  \Sigma_1   
~ &  B_1 ~ & O  ~&  O ~ & O ~ & O  ~&  O \\
O ~  &  A_2  ~  &   \Sigma_2 ~ &   B_2   ~&   O   ~  & O  ~ & O ~  &  O	   \\
O  & O ~   ~&  A_3  ~  &  \Sigma_3  ~   & B_3 ~  & O ~ &  O~ &  O  \\
O  & O ~  &  O ~&  A_4  ~  &  \Sigma_4  ~   & B_4 ~  & O ~  &  O  \\
O  & O ~  &  O ~& O ~ &  A_5  ~  &  \Sigma_5  ~   & B_5 ~  & O  \\
O  & O ~  &  O ~& O ~ &  O ~ & A_6  ~  &  \Sigma_6  ~   & B_6   \\
 B_7 ~   &  O ~ & O  ~ &  O ~& O ~ &   O ~ & A_7  ~  &  \Sigma_7    
\end{pmatrix}.
\end{equation}
Define the eight {\em tridiagonal blocks},
 $$\Sigma_0^{(c)}=\begin{pmatrix}
B_7    \\
\Sigma_0      \\
A_1  
\end{pmatrix},~ 
\Sigma_1^{(c)}=\begin{pmatrix}
B_0    \\
\Sigma_1      \\
A_2  
\end{pmatrix},~
\Sigma_2^{(c)}=\begin{pmatrix}
B_1    \\
\Sigma_2      \\
A_3  
\end{pmatrix},~ 
\Sigma_3^{(c)}=\begin{pmatrix}
B_2    \\
\Sigma_3      \\
A_4  
\end{pmatrix},$$
 $$\Sigma_4^{(c)}=\begin{pmatrix}
B_3    \\
\Sigma_4      \\
A_5  
\end{pmatrix},~ 
\Sigma_5^{(c)}=\begin{pmatrix}
B_4    \\
\Sigma_5      \\
A_6  
\end{pmatrix},~
\Sigma_6^{(c)}=\begin{pmatrix}
B_5    \\
\Sigma_6      \\
A_7  
\end{pmatrix},~
{\rm and}~
\Sigma_7^{(c)}=\begin{pmatrix}
B_6    \\
\Sigma_7      \\
A_0  
\end{pmatrix}.$$
Here $\Sigma_1^{(c)}$, $\Sigma_2^{(c)}$, $\Sigma_3^{(c)}$, $\Sigma_4^{(c)}$, 
$\Sigma_5^{(c)}$, 
 and $\Sigma_6^{(c)}$ are six blocks of 
 the matrix $M$ of (\ref{eqbidinv1}),
while  $\Sigma_0^{(c)}$ 
 and $\Sigma_7^{(c)}$ consist of two pairs of its blocks. Each pair, however, turns
 into a single block
if we  glue together the lower and upper boundaries
of the matrix $M$.
With 
 the diagonal block
$\Sigma_q$ and the
{\em tridiagonal block} $\Sigma_q^{(c)}$
we still associate a block column {\em } $M_q$ such that
$\mathcal C(M_q)=\mathcal C(\Sigma_q^{(c)})$.

The {\em admissible block} $N_q^{(c)}$, 
playing the role similar to that of a  neutered block column
of Definition \ref{defqs}, 
complements the  tridiagonal
block $\Sigma_q^{(c)}$ in its 
 block column. 
The block $N_q^{(c)}$ 
 is filled with zeros in the case of the matrix
$M$ of (\ref{eqbidinv1}) for every $q$,
$q=0,\dots,7$, but not so
in the case of general $8\times 8$
block matrix embedding the matrix $M$ of (\ref{eqbidinv1}).

Here are some sample unions 
of the tridiagonal blocks of the matrix
$M$ of (\ref{eqbidinv1}),
 $\Sigma_{0,1,\dots,7}^{(c)}=M$,
$$\Sigma_{0,1,2,3}^{(c)}=\begin{pmatrix}
B_7   &   O  &  O  &  O \\
\Sigma_0 ~ &  B_0~
   &  O~   &  O    \\
A_1  ~  &  \Sigma_1   
~ &  B_1 ~ & O    \\
O ~  &  A_2  ~  &   \Sigma_2 ~ &   B_2        \\
O  & O ~   ~&  A_3  ~  &  \Sigma_3     \\
O  & O ~  &  O ~&  A_4    
\end{pmatrix},~
\Sigma_{0,1}^{(c)}=\begin{pmatrix}
B_7  & O \\
\Sigma_0  &  B_0     \\
A_1  &  \Sigma_1      \\
O  &   A_2   
\end{pmatrix},~~{\rm and}~
\Sigma_{2,3}^{(c)}=\begin{pmatrix}
B_1  & O  \\
\Sigma_2  &  B_2     \\
A_3  &  \Sigma_3      \\
O  &   A_4   
\end{pmatrix}.$$
\end{example}

In {\bf Figure 5} the admissible blocks
are 
light gray or white;  
two adjacent blocks
of each  
black diagonal block are darker
gray;
the triples of these
black and 
gray blocks  
form the tridiagonal blocks.
The neutered block columns are either white or gray.


 \medskip

{\bf FIGURE  5}

\begin{figure}[ht!]\label{figure6}
\centering
\includegraphics[width=90mm]{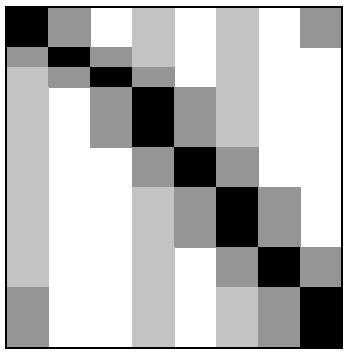}
\end{figure}

Let us generalize this  demonstration (see Figure 5).
Assume a block matrix $M$ with $k$ diagonal  blocks $\Sigma_q$,
of sizes $m_q^{(c)}\times n_q$, for $q=0,\dots,k-1$, and 
 glue  together its lower and upper block boundaries. Then
each diagonal block, including  
 the two extremal 
 blocks $\Sigma_0$ and $\Sigma_{k-1}$,  
has exactly two 
{\em adjacent blocks}
in its  block column: they are 
given by
the pair of the subdiagonal and 
superdiagonal blocks. 
Define the  
 {\em tridiagonal blocks}
$\Sigma_0^{(c)},\dots,\Sigma_{k-1}^{(c)}$
of sizes $m_q^{(c)}\times n_q$ 
by combining such triples of blocks where 
 $m_q^{(c)}=m_{q-1\mod k}+m_q+m_{q+1\mod k}$, $q=0,\dots,k-1$.
Write $m^{(c)}=\sum_{q=0}^{k-1}m_{q}^{(c)}$ and notice that 
$m^{(c)}=3m$ because
the number of rows in each of the three block diagonals 
sums to $m$. Therefore $s^{(c)}=\sum_{q=0}^{k-1}m_{q}^{(c)}n_q\le 
m^{(c)}n_+\le 3mn_+$.

 The complements 
of the tridiagonal blocks 
in their  block columns are also blocks,
 called 
 {\em admissible} (cf. \cite{B10}).
We call the matrix itself an {\em extended HSS matrix}, and
we extend accordingly our definitions of the
unions of blocks,
recursive and  
balanced merging,   
 $\rho$-neutered, 
balanced $\rho$-HSS,  
 $\rho$-HSS matrices, as well as
 $(\xi,\rho)$-neutered,
balanced $(\xi,\rho)$-HSS, and 
 $(\xi,\rho)$-HSS matrices
(cf.
Definitions \ref{defqs},
\ref{defunion}, and \ref{defhss}).
Can we extend 
 Theorem \ref{thext} and Corollary \ref{coext}
to the case of extended balanced $\rho$-HSS
matrices $M$ where we replace the integer parameters $m$
and $s$ by $m^{(c)}=3m$
and $s^{(c)}\le m^{(c)}n_+= 3mn_+$, respectively?
The extension of part (i) of Theorem \ref{thext} is immediate,
but in order to extend the algorithms supporting its part (ii), 
 we must impose some restriction on the input matrix $M$.

\begin{definition}\label{defrwc}
An extended
balanced  
$\rho$-HSS matrix 
is
{\em hierarchically regular}
if
all its diagonal blocks at the second  
factorization stage
of the associated balanced merging process
  have full rank. This matrix is 
{\em hierarchically well-con\-di\-tioned}
if these blocks are also well-con\-di\-tioned.
\end{definition}

\begin{theorem}\label{thalphbt}
Suppose that the matrix $M$ in Theorem \ref{thext}
is replaced by 
 an   extended $m\times n$
balanced $\rho$-HSS matrix $M^{(c)}$ 
and also suppose that
the integer parameters $m$ and $s$
in bounds (\ref{eqalph}) on $\alpha(M)$
and (\ref{eqbet}) on $\beta(M)$ are replaced by
$m^{(c)}=3m$ and
$s^{(c)}\le 3mn_+$, respectively.
Then bound (\ref{eqalph}) still holds, and 
 bound (\ref{eqbet})  holds
 if
 $m=n$ and if the matrix $M$
is hierarchically regular
and hierarchically well-conditioned.
\end{theorem}
\begin{proof}
Revisit the proof of the
Theorem \ref{thext}, by
replacing the integer parameters $m$ and $\bar s^{(j)}$ 
according to the assumptions of Theorem \ref{thalphbt},
and verify that  the proof still remains valid 
(use the assumption that the matrix $M$ is
 hierarchically regular and hierarchically well-conditioned
in order to extend bound (\ref{eqbet})). 
\end{proof}


\begin{corollary}\label{coext1}
Under the assumptions of 
Theorem \ref{thalphbt} suppose that  $k\rho=O(n)$
and $n_++\rho=O(\log(n))$. 
Then 
$\alpha(M)=O((m+n)\log^2(n))$ and 
$\beta(M)=O(n\log^3(n))$.
\end{corollary}



\bigskip

{\huge PART III: COMPUTATIONS WITH CV MATRICES 

\medskip 

\centerline {AND EXTENSIONS}}



\section{Approximation of the
CV and CV$^T$ 
matrices by HSS matrices
and algorithmic implications}\label{snrqs}




Our next goal is approximation of CV by HSS matrices,
which will imply fast approximate solution of   Problems 1--4
because in Part I we reduced them
to computations with CV matrices of (\ref{eqcvm}),
and  in Part II we described fast computations with HSS matrices.


\subsection{Small-rank approximation of certain Cauchy matrices}\label{ssra}


\begin{definition}\label{defss} (See  \cite[page 1254]{CGS07}.)
For a {\em separation bound} $\theta<1$ and a complex {\em separation center} $c$,
a pair of complex points $s$ and $t$
is $(\theta,c)$-{\em separated} 
if $|\frac{t-c}{s-c}|\le \theta$.
A pair of sets  of complex numbers  $\mathcal S$ and 
$\mathcal T$ is $(\theta,c)$-{\em separated} 
if  every pair of points  $s\in \mathcal S$
and $t\in \mathcal T$ 
is $(\theta,c)$-separated.
\end{definition}

\begin{lemma}\label{less} (See  \cite{R85} 
and \cite[equation (2.8)]{CGS07} or \cite{Pb}.)
Suppose a pair of complex points $s$ and $t$ is
 $(\theta,c)$-separated
for $0\le \theta<1$ and a complex {\em center} $c$.
 Fix a positive integer $\rho$ 
and write $q=\frac{t-c}{s-c}$ and
$|q|\le \theta$. 
Then  
$\frac{1}{s-t}=
\frac{1}{s-c}\sum_{h=0}^{\rho-1}\frac{(t-c)^h}{(s-c)^h}+\frac{q_{\rho}}{s-c}$
for
$|q_{\rho}|=\frac{|q|^{\rho}}{1-|q|}\le \frac{\theta^{\rho}}{1-\theta}$
and~ a positive integer $\rho$.
\end{lemma}

\begin{corollary}\label{coss}  (Cf. \cite[Section 2.2]{CGS07},
 \cite{B10}, or \cite{Pb}.)
Suppose that two sets 
of $2n$ distinct complex numbers
$\mathcal S=\{s_0,\dots,s_{m-1}\}$ and 
$\mathcal T=\{t_0,\dots,t_{n-1}\}$ 
  are $(\theta,c)$-separated from one another
for $0<\theta<1$ and a {\em global complex center} $c$.
Define the Cauchy matrix
$C=(\frac{1}{s_i-t_j})_{i,j=0}^{m-1,n-1}$
and let  $\delta=\delta_{c,\mathcal S}=\min_{i=0}^{m-1} |s_i-c|$
denote the distance from the center $c$ to the set $\mathcal S$.
 Fix a positive integer $\rho$
and define the $m\times \rho$ matrix
$F=(1/(s_i-c)^{\nu+1})_{i,\nu=0}^{m-1,\rho-1}$
and the $n\times \rho$ matrix
$G=((t_j-c)^{\nu})_{j,\nu=0}^{n-1,\rho-1}$.
(We can compute these matrices by using
 $(m+n)\rho+m$ arithmetic operations.)
Then
\begin{equation}\label{eqappr}
 C=FG^T+E,~|E|\le \frac{\theta^{\rho}}{(1-\theta)\delta}.
\end{equation}
\end{corollary}



\subsection{Block partition of a Cauchy matrix}\label{slc}


Generally neither CV matrix of equation (\ref{eqcvm}) nor its blocks of a large size
 have global separation centers. 
So, instead of the approximation of a CV matrix by a low-rank matrix,
we seek its approximation by an extended balanced $\rho$-HSS matrix 
for a bounded integer  $\rho$. 
At first we fix a reasonably large integer $k$ and then partition the complex plane into $k$
congruent sectors sharing the origin 0. The following definition induces 
a uniform {\em $k$-partition} of the knot sets
$\mathcal S$ and $\mathcal T$ and thus induces a block partition of the associated
Cauchy matrix. In the next subsection we specialize these partitions   
to the case of a CV matrix.




\begin{definition}\label{defasd} (See Figure 6.)
$\mathcal A(\phi,\phi')=\{z=\exp(\psi\sqrt {-1}):~0\le \phi\le \psi<\phi'< 2\pi\}$
is the 
{\em semi-open 
arc} of the unit circle 
$\{z:~|z| =1\}$ 
with length $\phi'-\phi$ and endpoints 
$\tau=\exp(\phi\sqrt{-1})$ and $\tau'=\exp(\phi'\sqrt{-1})$.
$\Gamma(\phi,\phi')=
\{z=r\exp(\psi\sqrt {-1}):~r\ge 0,~0\le \phi\le \psi<\phi'< 2\pi\}$
is the 
{\em semi-open sector}. $\bar {\Gamma}(\phi,\phi')$ is its  exterior.
\end{definition}

In {\bf Figure 6} we mark by black color an arc of the unit circle
$\{z:~|z=1|\}$.
The five line intervals
 $[0,\tau]$, $[0,c]$, $[0,\tau']$, $[\tau,c]$, and $[c,\tau]$
are shown by 
dotted lines.
Two broken lines represent
the two line intervals
 bounding the intersection
of the sector $\Gamma (\psi,\psi')$ and the unit disc
 $D(0,1)=\{z:~|z|\le 1\}$. The two perpendiculars
from the center $c$ onto these two bounding line intervals
are also represented by broken lines.

 \medskip

{\bf FIGURE  6}

\begin{figure}[ht!]\label{figure10}
\centering
\includegraphics[width=90mm]{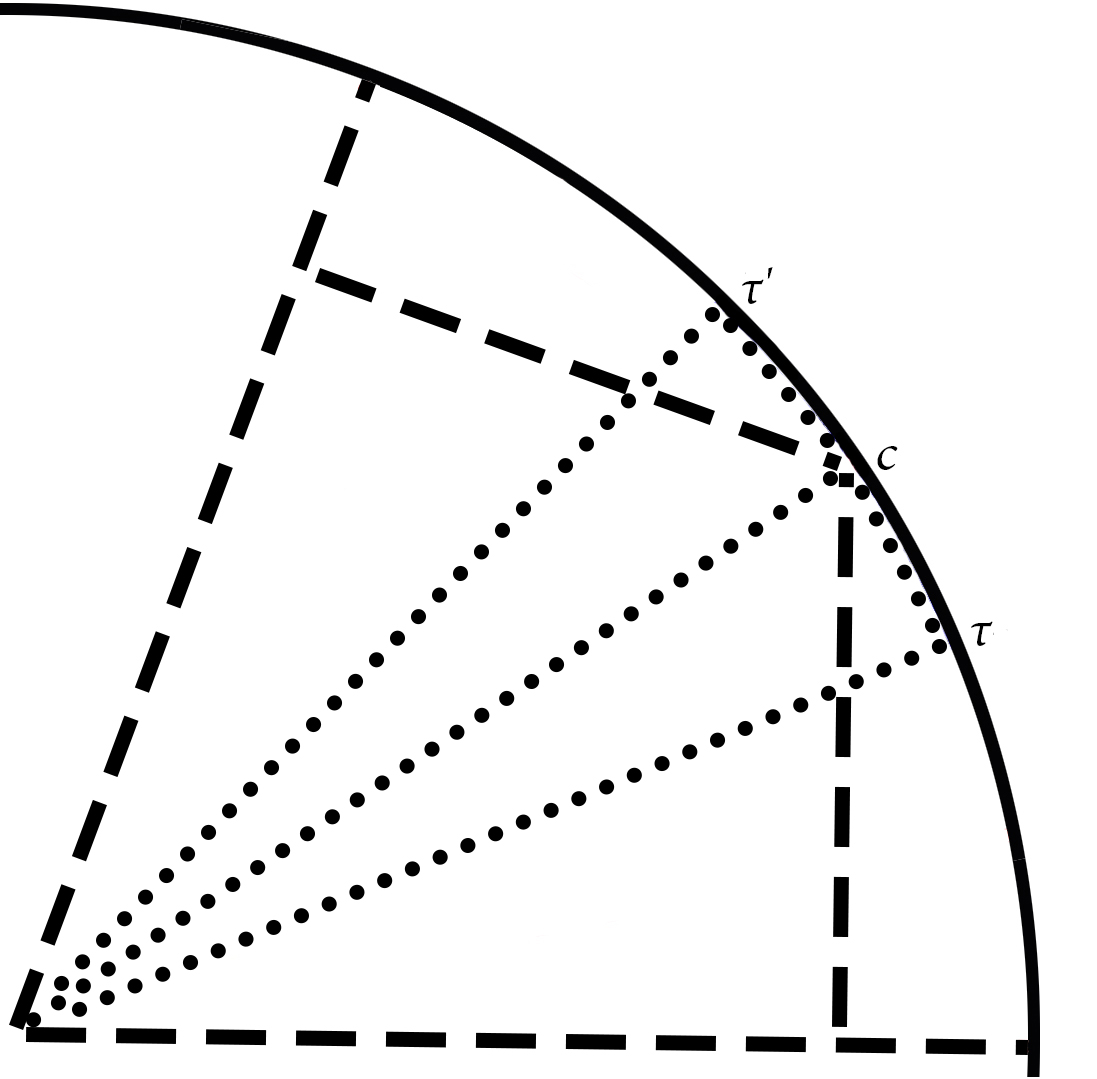}
\end{figure}





 Fix a positive  integer  $l_+$, write $k=2^{l_+}$, 
$\phi_q=2q\pi/k$, and
$\phi'_{q}=\phi_{q+1\mod k}$. Then
$|\phi'_{q}-\phi_{q}|=2\pi/k~{\rm for~all}~q$.
 
Partition
 the unit circle $\{z:~|z =1|\}$ 
 by $k$ equally spaced points $\phi_0,\dots,\phi_{k-1}$
 into $k$ semi-open arcs 
$\mathcal A_q=\mathcal A(\phi_q,\phi'_{q})$, each 
of length $2\pi/k$.
Define 
the semi-open sectors 
$\Gamma_q=\Gamma(\phi_q,\phi'_{q})$
for $q=0,\dots,k-1$, that is, $\Gamma_q=\Gamma(\phi_q,\phi_{q+1})$,
for $q=0,\dots,k-2$,  and $\Gamma_{k-1}=\Gamma(\phi_{k-1},\phi_{0})$. 

Assume the polar representation 
 $s_i=|s_i|\exp(\mu_i\sqrt{-1})$ and 
 $t_j=|t_j|\exp(\nu_j\sqrt{-1})$.
 
Notice that the knots
$t_0,\dots,t_{n-1}$ have been enumerated 
in the counter-clockwise order of the  angles  $\nu_j$, beginning with the knots in
the sector $\Gamma(\phi_0,\phi_0')$. 
Similarly 
re-enumerate the knots 
$s_0,\dots,s_{m-1}$, in the
counter-clockwise order of the  angles  $\mu_j$.  
Induce
the  block  partition
of a Cauchy matrix $C=(C_{p,q})_{p,q=0}^{k-1}$
and its  partition 
 into    block columns 
$C=(C_{0}~|~\dots|~C_{k-1})$
such that 
$$C_{p,q}=\Big(\frac{1}{s_i-
t_j}\Big)_{s_i\in \Gamma_p,t_j\in \Gamma_q}~{\rm and}~C_{q}=\Big(\frac{1}{s_i-
t_j}\Big)_{s_i\in \{0,\dots,n-1\},t_j\in \Gamma_q}~{\rm  for}~p,q=0,\dots,k-1.$$
Furthermore, for every $q$,	 define
(i) the diagonal block
$\Sigma_q=C_{q,q}$, 
(ii) the two  adjacent blocks
$C_{q-1\mod k,q}$ and $C_{q+1\mod k,q}$ above and below it,
(iii) the tridiagonal block
$\Sigma_q^{(c)}$ (made up of the block $C_{q}$ and the two  adjacent blocks), 
and 
(iv) the admissible block $N_q^{(c)}$, which
 complements the tridiagonal block
$\Sigma_q^{(c)}$ in its  block column $C_q$.  
  
If a tridiagonal block $\Sigma_q^{(c)}$ is empty, then 
the admissible block  $N_q^{(c)}$ occupies the entire  block column $C_q$,
that is, this block column has rank at most $\rho$. 
If, on the contrary, a tridiagonal block $\Sigma_q^{(c)}$
occupies the entire  block column $C_q$, then
only the tridiagonal blocks in the two neighboring  block columns
$C_{q-1\mod k}$ and $C_{q+1\mod k}$
can be nonempty, and so all the other   block columns
are occupied entirely by admissible blocks and 
hence have ranks  at most $\rho$.


\subsection{Separation of the tridiagonal and admissible blocks of a CV matrix}\label{ssep}


The following lemma can be readily verified (cf. Figure 6).

\begin{lemma}\label{learc} 
$0\le \chi\le \phi\le \eta <\phi'< \chi'\le \pi/2$
and write 
$\tau=\exp(\phi\sqrt{-1})$, $c=\exp(\eta\sqrt{-1})$, and $\tau'=\exp(\phi'\sqrt{-1})$.
Then  $|c-\tau|=2\sin (\frac{\eta-\phi}{2})$
and  the distance from the point $c$ to the sector $\bar \Gamma(\chi,\chi')$
is equal to $\sin (\psi)$, for $\psi=\min\{\eta-\chi,\chi'-\eta\}$.
\end{lemma}

Next we specialize the block partition 
of the previous subsection to the case of a CV matrix 
$C_{{\bf s},f}$ of (\ref{eqcvm}) for a fixed complex $f$
such that $|f|=1$. In this case
 $t_j=f\omega_k^j$ for  $\omega_k=\exp(2\pi\sqrt {-1}/k)$, 
$j=0,\dots,n-1$, and
every arc $\mathcal A_q$ contains $\lceil n/k\rceil$ 
or $\lfloor n/k\rfloor $ knots $t_j$.

In Figure 7, $\psi=\phi_1+\frac{\phi_0}{2}$. 





 \medskip

{\bf FIGURE 7}
\begin{figure}[ht!]\label{figure7}
\centering
\includegraphics[width=130mm]{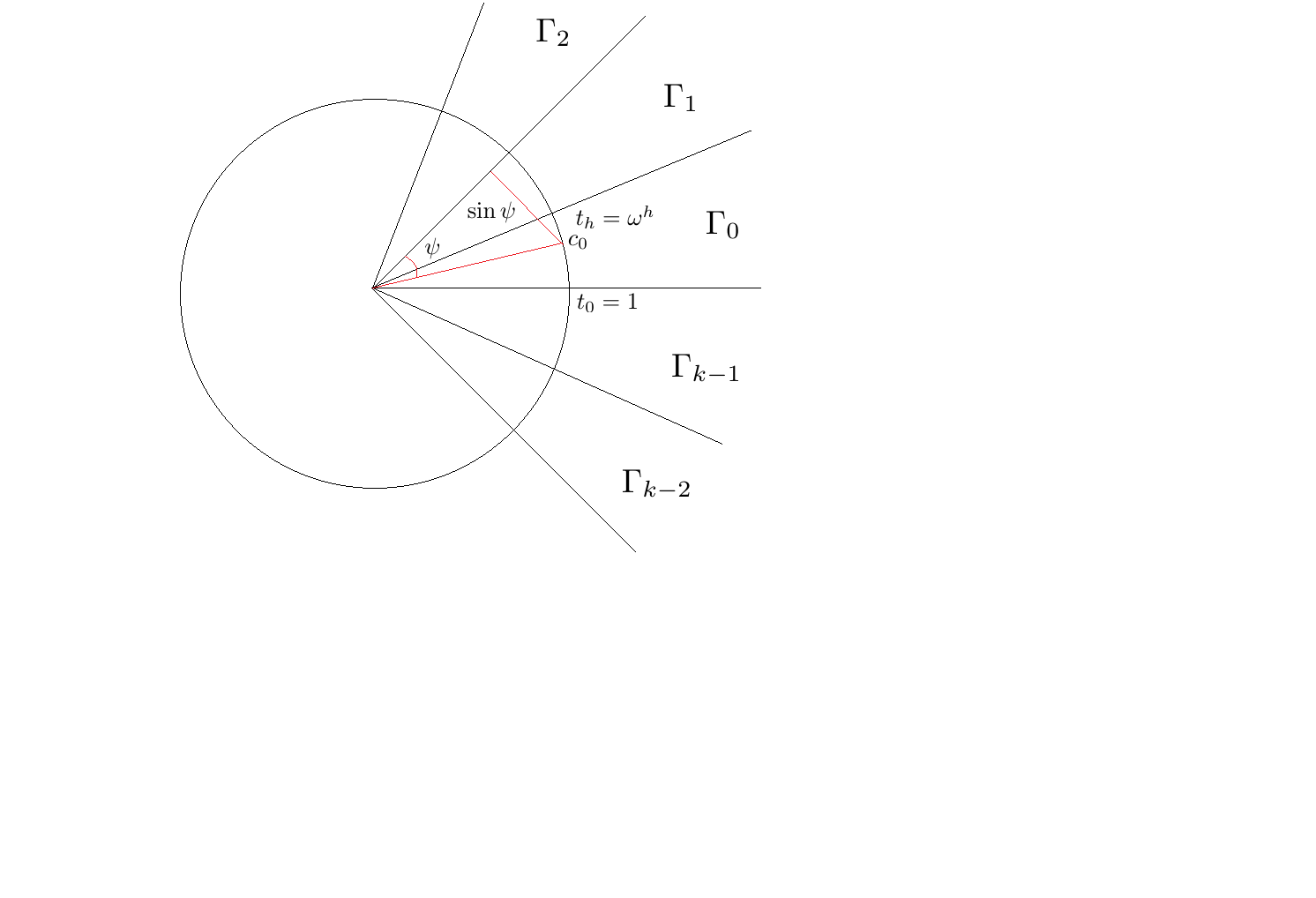}
\end{figure}

\begin{theorem}\label{thsep} 
(Cf. Figure 7.)
Assume a uniform  $k$-partition of the knot sets of a CV matrix above
for $k\ge 12$. 
Let $\Gamma'_q$ denote the union of the sector 
$\Gamma_q$ and its two  adjacent sectors on both sides, that is,  
$\Gamma'_q=\Gamma_{q-1 \mod k}\cup\Gamma_{q}\cup \Gamma_{q+1 \mod k}$.
Write  $\bar \Gamma'_q$ to denote the exterior of the sector $\Gamma'_q$
and write  $c_q$ to denote the midpoints of the arcs  
 $\mathcal A_q=\mathcal A(\phi_q,\phi_q')$
for $\phi_q'=\phi_{q+1\mod k}$ and $q=0,\dots,k-1$. 
Furthermore let $\bar\delta_q$ denote the distance from the center $c_q$
to the sector $\bar \Gamma'_q$. 
Then, for every $q$,  
(i) $\bar\delta_q\ge |\sin (\frac{3\pi}{k})|$ and 
(ii)  the arc $\mathcal A_q$
and the sector $\bar \Gamma'_q$ are $(\theta,c_q)$-separated
for  $\theta=2\sin(\frac{\pi}{2k})/\sin (\frac{3\pi}{k})$.
\end{theorem}

\begin{proof}
Suppose that $1\le q\le k-3$. 
Then
$\Gamma'_{q}=\Gamma(\phi_{q-1},\phi_{q+2})$. 
Apply Lemma \ref{learc}, for $\chi=\phi_{q-1}$, $\phi=\phi_q$, $c=c_q$,  $\phi'=\phi'_q=\phi_{q+1}$,
and $\chi'=\phi_{q+2}$,
and obtain the theorem. Similarly prove the theorem in the cases where
 $q=0$, $\Gamma'_0=\Gamma(\phi_{k-1},\phi_{2})$; $q=k-2$
and $\Gamma'_{k-2}=\Gamma(\phi_{k-3},\phi_{0})$, and $q=k-1$
and $\Gamma'_{k-1}=\Gamma(\phi_{k-2},\phi_{1})$.
\end{proof}

Recall that $\sin (y)\approx y$ as $y \approx 0$,
and therefore $\theta\approx 1/3$ provided
 that  the integer 
$k$ is large.
Notice that for every $q$ the admissible block $N_q^{(c)}$
is defined by the knots $t_j$ lying on the arc  $\mathcal A_q$
and the knots $s_i$ lying in the sector $\bar \Gamma'_q$, and
apply Corollary \ref{coss}. For every $q$, $q=0,\dots,k-1$,
 write
$\delta_q=\min_{s_i\in \bar \Gamma'_q} |s_i-c_q|$, then notice that 
$\delta_q\ge \bar\delta_q$,
 and obtain the following result.

\begin{corollary}\label{colrba1}
Assume a sufficiently large integer $k$, $2k<n$,  
and let a uniform $k$-partition of the knot sets $\mathcal S$ and $\mathcal T$
of  an $m\times n$ CV matrix $C$
define $k$ admissible blocks
$N_0^{(c)},\dots,N_{k-1}^{(c)}$.
Then all of them  have the $|E|$-ranks
at most $\rho$,
that is, $C$ is an extended 
$(|E|,\rho)$-neutered matrix,
where $|E|$ and $\rho$
satisfy  bound (\ref{eqappr}) for $\theta\approx 1/3$ and 
$\delta=\min_{q=0}^{k-1}|\delta_q|\ge |\sin (\frac{3\pi}{k})|$.
\end{corollary}

Our $k$-uniform partition of the complex plane 
into $k$ congruent sectors defines a desired partition of 
 CV matrix into $(\theta,c_q)$-separated blocks for $\theta\approx 1/3$
or smaller. 
Trying to extend our results to the more 
general class of Cauchy matrices $C_{\bf s,t}$ whose all knots
$t_j$ lie on the unit circle $\{z:~|z|=1\}$, one may 
consider various other partitions of the complex plane
and apply the following 
extension of Lemma \ref{learc}
and Theorem \ref{thsep}. 

\begin{lemma}\label{lelrba0} 
Assume
the numbers   
$\theta$,
$\phi$, $\phi'$, and $c$ such that 
 $0<\theta<1$, 
 $0\le \phi<\phi'\le 2\pi$,
and $c=\exp(0.5(\phi'+\phi)\sqrt{-1})$ is the midpoint 
of  the arc 
$\mathcal A(\phi,\phi')$. 
Write $r=r(\phi,\phi',\theta)=\frac{2}{\theta}\sin(\frac{\phi'-\phi}{4})$.
Let $D(c,r)=\{z:~|z-c|\le r\}$
denote the disc on the complex plane 
with a center $c$ and a radius $r$ and let 
$\bar D(c,r)=\{z:~|z-c|> r\}$
denotes the exterior of this disc.
Then the two sets 
$ \mathcal A(\phi,\phi')$
and 
$\bar D(c,r)$
are $(\theta,c)$-separated.
\end{lemma}



\subsection{Approximation of a CV matrix by a balanced $\rho$-HSS matrix
and the complexity of   
 approximate computations with CV 
matrices} \label{sfnft} 


Let $\delta^{(h)}$ denote the minimum distance from the centers $c_q$
to the knots $s_i$ lying in the admissible blocks after the $h$th 
recursive merging.
Recall that the  angles $2\pi/k$ of the $k$  congruent sectors
$\Gamma_0,\dots,\Gamma_{k-1}$
are recursively doubled
in every merging. So Lemma \ref{learc} implies that
$\delta^{(h)}\ge \sin(3\pi 2^h/k)$ after the $h$th merging,
$h=1,\dots,l$.
We define the recursive merging by choosing
 the integers $k=2^{l_+}$ and $l<l_+$. Choose them  
such that $k/2^l=2^{l_+-l}\ge 6$.
Then 
$\delta^{(h+1)}>\delta^{(h)}>\delta^{(0)}\ge \delta_-=\sin(\frac{3\pi}{k})$
for~all $h$,
and so $ \delta_-\approx \frac{3\pi}{k}$ for large integers $k$.
Together with  Corollary \ref{colrba1} 
these relationships imply the following result.

\begin{theorem}\label{thbal}
The CV matrix $C$ of Corollary \ref{colrba1} as well as its transpose 
 CV$^T$ matrix $C^T$ are two
extended balanced $(\xi,\rho)$-HSS matrices
where the values $\xi$ and $\rho$ are linked by 
bound (\ref{eqappr}) for 
$|E|=\xi$,
$\theta=2\sin(\frac{\pi}{2k})/\sin(\frac{3\pi}{k})$,
 and
$\delta=\delta_h\ge \delta_-=\sin(\frac{3\pi}{k})$,
so that $\theta \approx 1/3$ and $\delta_-\approx \frac{3\pi}{k}$,
for large integers $k$.
\end{theorem}

Combine Corollary \ref{coext1} with  this theorem applied for 
$k=2^{l_+}$ of order $n/\log(n)$, for
$\rho$ and $\log(1/\xi)$ of order $\log(n)$,
and for  $l<l_+$ such that $l_+-l\ge 6$ 
(verify that in this case the assumptions of the corollary are satisfied), and
obtain the following complexity estimates for  CV matrices $C$ and CV$^T$ matrices $C^T$.




\begin{theorem}\label{thalphbtc}
Assume an $m\times n$ CV matrix $C$
and a positive $\xi$
such that $\log(1/\xi)=O(\log(n))$. 
Then
$\alpha_{\xi}(C)=O((m+n)\log^2(n))$.
If in addition
$m=n$ and if the matrix $C$ is $\xi$-approximated by
a hierarchically regular
extended 
 balanced $\rho$-HSS matrix,
then $\beta_{\xi}(C)=O(n\log^3(n))$.
The same bounds hold for the CV$^T$ matrix $C^T$ replacing $C$.
\end{theorem}


\section{Extensions and implementation}\label{scmplcv} 


\subsection{Computations with
 matrices having displacement  structure,
polynomials, and rational functions
}\label{scmplcv1} 


 By combining the algebraic techniques of transformation of matrix structure
of \cite{P90} with the  FMM/HSS techniques,  \cite[Section 9]{P15}
 extends 
the complexity bounds of Theorems \ref{thalphbtc} and  \ref{thcvv1} to
generalized Cauchy matrices $M=(f(s_i-t_j))_{i,j=0}^{m-1,n-1}$
for various functions $f(z)$ such as $z^{-p}$ for a positive integer $p$,
$\ln z$, and $\tan z$,
 to $n\times n$ structured matrices $M$
 having the displacement structures of 
 Toeplitz, Hankel, 
Cauchy and Van\-der\-monde types (cf. also \cite{Pb}),
and in particular to
 Cauchy matrices $M=C_{\bf s,t}$ having arbitrary sets of knots $\mathcal S$ and $\mathcal  T$.
In the latter case  the approximation error bound $\xi$
increases by a factor 
bounded from above by the condition  number $\kappa(M)=||M||~||M^{+}||$,
and the results are readily extended to Problems 3 and 4 of 
multipoint rational evaluation and interpolation.
Next we specify the simpler extension to
computations with a Van\-der\-monde matrix, its transpose, and polynomials.

\begin{theorem}\label{thcvv1}
For a positive $\xi$ and a vector
${\bf s}=(s_i)_{i=0}^{m-1}$, write 
$V=V_{\bf s}$ 
and $s_+=\max_{i=0}^{m-1}|s_i|$.  \\ 
(i) Then 
$\alpha_{\xi}(V)+\alpha_{\xi}(V^T)=
O((m+n)\rho\log^2(n))$
provided that $s_+$ is bounded from above by a constant.  \\
(ii) Suppose that,  
 for $m=n$ and some complex $f$, $|f|=1$, the CV matrix
 $C_{{\bf s},f}$ 
has been $\xi$-approximated by
a hierarchically nonsingular extended 
 balanced $(\xi,\rho)$-HSS matrix. Then 
$\beta_{\xi}(V)+\beta_{\xi}(V^T)=O(n\rho^3\log(n))$.  \\
(iii) One can extend the above bounds 
on $\alpha_{\xi}(V)$ and $\beta_{\xi}(V)$
to the solution of
Problems 1 and 2
of Section \ref{s3}.
\end{theorem}
\begin{proof}
With no loss of generality we can assume that $m=n$.
Combine Theorem \ref{thalphbtc},
equations 
(\ref{eqvs1}), (\ref{eqvs-1})  and their transposes. 
The matrices  $\diag(\omega^{j})_{j=0}^{n-1}$,
$\Omega/\sqrt n$, $\Omega^H/\sqrt n$, and
$\diag(f^{j})_{j=0}^{n-1}$
 are unitary, and 
so multiplication by them and by their inverses makes 
no impact on the output error norms.
Multiplication by the matrix $\diag(s_i^n-f^n)_{i=0}^{n-1}$
can increase the value $\xi$ 
 by at most a factor of $1+s_+^n\le 1+|V_{\bf s}|s_+
$,
while multiplication by the inverse of this matrix 
increases $\xi$ by a factor of 
$\Delta=1/\max_{f:~|f|=1}\min_{i=0}^{n-1}|s_i^n-f^n|$,
which is  at most 
 $2n$ 
for a proper choice of the value $f$
such that $|f|=1$. Then the increase
by a factor of $\Delta$ would
 make no impact on the asymptotic bounds of 
Theorem \ref{thcvv1}, and so we complete
 the proof of parts (i)
and (ii). Equations 
of Problem 1
extend the proof to part (iii).
\end{proof}


\subsection{Simplified implementation
}\label{ssmplimpl}


One can implement our algorithms
by computing the centers
$c_q$ and the admissible
blocks $\widehat N_q$ of bounded ranks
in the merging process,  
but can avoid a large part of the
computations
by following the recipe of the papers \cite{CGS07},
\cite{X12},
 \cite{XXG12}, and
 \cite{XXCB14}.
The idea is to bypass the computation of the centers
$c_q$ and immediately
 compute HSS generators for the
 admissible blocks $\widehat N_q$,
defined by HSS trees. The length (size) of the generators at every merging stage
(represented by a fixed level of the tree) can be 
chosen equal to the available  
upper bound on the numerical ranks of these blocks 
or can be adapted empirically.  
See \cite[Section 10.1]{PLSZa} for a recent acceleration of this stage.


\bigskip

\noindent{\huge PART IV: NUMERICAL TESTS AND CONCLUSIONS}



\section{Numerical Experiments}\label{ststs}

Numerical experiments
	have been performed under our supervision 
in the Graduate Center of the City University of New York
by Franklin Lee and Aron Wolinetz
(Section \ref{sexprnk}) and
		by Liang Zhao (Section \ref{sexppoly}).
All computations have been performed with the IEEE standard double precision.
The codes are available upon request.


\subsection{Experimental computation of numerical ranks of the admissible blocks 
of CV matrices}\label{sexprnk}

			The test programs were written
				in Python 3.3.3,
				using the Numpy 1.7.1, Scipy 0.12.1, and Sympy 0.7.3 libraries.
			The tests were run
				on Windows 7 64-bit SP1
				on a Toshiba Satellite L515-S4925
					with a Pentium Dual-Core T4300 @ 2.10GHz x2 processor.
Random numbers were generated uniformly with the language's Mersenne twister over the range
 $\{x:0 \leq x < 1\}$
 and
  extended  to the ranges  $\{y:~a \leq y < b\}$
		for $y=a+(b-a)x$.

For $n=1024, 2048, 4096$
	we computed 
		the vectors $(\omega_j)_{i=0}^{n-1}$
			of the $n$th roots of unity,
	and
		for every pair of $n$ and $h$, $h=0,1,4$,
			we generated 100,000 instances 
					of complex numbers $s_0,
						\dots,s_{n-1}$,
		thus
		defining
			$n\times n$ CV matrices
				$C_{{\bf s},1}=(\frac{1}{s_i-\omega^j})_{i,j=0}^{n-1}$.

We generated
	the knots
			$s_i=|s_i|\exp(\phi_i\sqrt {-1})$
		as follows.
	At  first we
		generated the
			angles $\bar \phi_i$
		over the range  $0\le \bar\phi_i<2\pi$ and
			the values $|s_i|$ over the range $[1-1/2^h,1+1/2^h)$
				for $h=0,1,4$
				and $i=0,\dots,n-1$,
		in all cases
			independently for all $i$ and $t$.
	Then for every vector $(\bar\phi_i)_{i=0}^{k-1}$ we computed the
		permutation matrix $P$ defining
			the  vector $(\phi_i)_{i=0}^{n-1}=P(\bar\phi_i)_{i=0}^{n-1}$
				with the coordinates $\phi_0,\dots,\phi_{n-1}$ 
					in the nondecreasing order.		
For every pair of the vectors
					$(|s_i|)_{i=0}^{n-1}$ and $(\phi_i)_{i=0}^{n-1}$ we 
				defined the vector
					$(s_i)_{i=0}^{n-1}=(|s_i|\exp(\phi_i\sqrt {-1}))_{i=0}^{n-1}$ and
				the CV matrix $C=(\frac{1}{s_i-\omega^j})_{i,j=0}^{n-1}$.
	Then we
		fixed the integers $k=4,32,512,2048$, skipped integer pairs
			$(k,n)$ where $k<2$ or $n/k<2$, and
		defined tridiagonal and
			admissible blocks
		by following  the recipes of Section
			\ref{snrqs}.

Finally we fixed the
tolerances $\xi=10^{-q}$ for $q=2,3,4$
and 
computed
the $\xi$-ranks of nonempty
admissible blocks $N_q^{(c)}$
by applying the rank function 
		 $\text{numpy.linalg.matrix\_rank}(X, tol)$.

Tables \ref{tab0}--\ref{tab4} show the average computed values of 
the $\xi$-ranks in these tests.
They vary rather little, remaining consistently small, when we changed the parameters $h$, 
$k$, and $\xi$,  and they grew very slowly when we doubled the 
matrix dimension $n$.

We also computed the average norms of the admissible blocks. They
ranged between 100 and 1000.

\bigskip

\begin{table}[h]
  \caption{The $\xi$-ranks of the admissible blocks for $h=0$}
  \label{tab0}
  \begin{center}
    \begin{tabular}{| c | c | c | c | c| }
    \hline
$\xi$ & \textbf{$n$}&k=4&k=32&k=512 \\ \hline
0.01 & $1024$ & $5.0$ & $5.0$ & $2.0$   \\ \hline
0.01 & $2048$ & $5.0$ & $5.0$ & $3.0$   \\ \hline
0.01 & $4096$ & $5.0$ & $5.0$ & $3.8$  \\ \hline

0.001 & $1024$ & $6.0$ & $6.0$ & $2.0$   \\ \hline
0.001 & $2048$ & $6.0$ & $6.0$ & $3.8$   \\ \hline
0.001 & $4096$ & $6.0$ & $6.3$ & $4.3$  \\ \hline

0.0001 & $1024$ & $7.0$ & $7.0$ & $2.0$   \\ \hline
0.0001 & $2048$ & $7.0$ & $7.0$ & $4.0$   \\ \hline
0.0001 & $4096$ & $7.0$ & $7.8$ & $5.0$  \\ \hline
  \end{tabular}
  \end{center}
\end{table}


\begin{table}[h]
  \caption{The $\xi$-ranks of the admissible blocks for $h=1$}
  \label{tab1}
  \begin{center}
    \begin{tabular}{| c | c | c | c | c| }
    \hline
$\xi$ & \textbf{$n$}&k=4&k=32&k=512 \\ \hline
0.01 & $1024$ & $4.0$ & $5.0$ & $2.0$   \\ \hline
0.01 & $2048$ & $4.0$ & $5.0$ & $3.4$   \\ \hline
0.01 & $4096$ & $5.0$ & $5.8$ & $4.0$  \\ \hline

0.001 & $1024$ & $5.0$ & $6.0$ & $2.0$   \\ \hline
0.001 & $2048$ & $5.0$ & $6.0$ & $4.0$   \\ \hline
0.001 & $4096$ & $6.0$ & $7.0$ & $4.8$  \\ \hline

0.0001 & $1024$ & $6.0$ & $7.0$ & $2.0$   \\ \hline
0.0001 & $2048$ & $6.0$ & $7.0$ & $4.0$   \\ \hline
0.0001 & $4096$ & $6.0$ & $8.0$ & $5.4$  \\ \hline
  \end{tabular}
  \end{center}
\end{table}

\begin{table}[h]
  \caption{The $\xi$-ranks of the admissible blocks for $h=4$}
  \label{tab4}
  \begin{center}
    \begin{tabular}{| c | c | c | c | c| }
    \hline
$\xi$ & \textbf{$n$}&k=4&k=32&k=512 \\ \hline
0.01 & $1024$ & $4.0$ & $5.0$ & $2.0$   \\ \hline
0.01 & $2048$ & $4.0$ & $5.0$ & $4.0$   \\ \hline
0.01 & $4096$ & $4.0$ & $5.0$ & $5.0$  \\ \hline

0.001 & $1024$ & $4.0$ & $6.0$ & $2.0$   \\ \hline
0.001 & $2048$ & $5.0$ & $6.0$ & $4.0$   \\ \hline
0.001 & $4096$ & $5.0$ & $6.0$ & $5.9$  \\ \hline

0.0001 & $1024$ & $5.0$ & $7.0$ & $2.0$   \\ \hline
0.0001 & $2048$ & $5.0$ & $7.0$ & $4.0$   \\ \hline
0.0001 & $4096$ & $6.0$ & $7.0$ & $6.6$  \\ \hline
  \end{tabular}
  \end{center}
\end{table}




\subsection{Multipoint numerical evaluation of polynomials}\label{sexppoly}

We tested  numerical behavior of our algorithms for 
approximate evaluation of 
 real 
and complex Gaussian random polynomials $p(x)$ of degree 
$n-1$, for $n = 64, 128, 256, 512, 1024, 2048, 4096$,
and  generated the knots of the evaluation lying 
 in the unit disc 
$\{z:~|z|\le 1\}$.

We performed the tests on a Dell server running Windows system 
and using MATLAB R2014a with double precision. 
 We applied the MATLAB function ''randn()'' in order to generate 
the real polynomial coefficients and
the 
real and imaginary parts separately for the complex coefficients.
   
The knots of the evaluation,  
$s_i = r\times\exp(2\pi \theta \sqrt {-1})$,
depended on two parameters $r$ and  $\theta$.
 In all tests
 we defined
 the values $\theta$   
by applying the uniform random number generator 
''rand()'' to the line interval $[0, 1)$,
and we generated the absolute values $r$ in two ways.
 
In one series of our tests we set the absolute value $r$  to 1, 
thus placing the knots $s_i$ onto the unit circle $\{x:~|x|=1\}$,
and then we displayed the test
results 
in Tables \ref{PolyEvalReal} and \ref{PolyEvalComplex}.

In another series of our tests we  generated the  absolute value $r$
at random by applying the uniform random number generator ''rand()'' 
to the line interval $[0, 1]$, 
and then we displayed the test
results 
in Tables \ref{PolyEvalReal2} and \ref{PolyEvalComplex2}.
The latter tests cover polynomial evaluation at the knots lying in 
the unit disc  $\{z:~|z|\le 1\}$, but can be extended to the
evaluation outside it, by shifting from a polynomial $p(x)$ of degree $n$
to the reverse polynomial $x^np(1/x)$.

In all tables  
the columns ``Max. Rank`` represent the maximum 
$\xi$-ranks of the off-tridiagonal blocks in 
the computation, for $\xi = 10^{-5}$. 
The columns ``Error'' represent the absolute difference of our 
computed values of the polynomials and the output 
of the MATLAB function  ''polyval()" for the same inputs. 

All tests have been repeated  100 times for each $n$ and  the average results
 have been displayed.

According to the test results, the computed maximum numerical rank was consistently low, implying that our algorithm ran fast, 
even though it still produced quite accurate output values.

For comparison, Table \ref{SuperfastEval} displays the mean values and standard deviations
of the output errors observed in 
our test of the polynomial evaluation algorithm of \cite{MB72}
applied to the same inputs and also with the IEEE standard double precision.
According to these results, the algorithm has
consistently performed with much inferior output accuracy for polynomials of degree 32 and higher. 

\begin{table}[h]
\caption{Evaluation of Real Gaussian Polynomials on the Unit Circle}
\label{PolyEvalReal}
\begin{center}
\begin{tabular}{|c|c|c|}
\hline
Degree	&	Max. Rank	&	Error	\\ \hline
32	&	13	&$	6.60\times 10^{-07	}$\\ \hline
64	&	11	&$	8.05\times 10^{-08	}$\\ \hline
128	&	12	&$	5.88\times 10^{-07	}$\\ \hline
256	&	12	&$	4.01\times 10^{-07	}$\\ \hline
512	&	12	&$	2.27\times 10^{-07	}$\\ \hline
1024	&	12	&$	5.77\times 10^{-08	}$\\ \hline
2048	&	13	&$	1.38\times 10^{-06	}$\\ \hline
4096	&	13	&$	2.99\times 10^{-05	}$\\ \hline

\end{tabular}
\end{center}
\end{table}

\begin{table}[h]
\caption{Evaluation of Real Gaussian Polynomial in the Unit Disk}
\label{PolyEvalReal2}
\begin{center}
\begin{tabular}{|c|c|c|}
\hline
Degree	&	Max. Rank	&	Error	\\ \hline
32	&	18	&$	1.90\times 10^{-06	}$\\ \hline
64	&	13	&$	1.47\times 10^{-06	}$\\ \hline
128	&	13	&$	1.13\times 10^{-06	}$\\ \hline
256	&	12	&$	9.09\times 10^{-07	}$\\ \hline
512	&	13	&$	7.05\times 10^{-07	}$\\ \hline
1024	&	12	&$	5.49\times 10^{-07	}$\\ \hline
2048	&	13	&$	4.67\times 10^{-07	}$\\ \hline
4096	&	13	&$	3.80\times 10^{-07	}$\\ \hline
\end{tabular}
\end{center}
\end{table}

\begin{table}[h]
\caption{Evaluation of Complex Gaussian Polynomials on the Unit Circle}
\label{PolyEvalComplex}
\begin{center}
\begin{tabular}{|c|c|c|}
\hline
Degree	&	Max. Rank	&	Error	\\ \hline
32	&	12	&$	5.68\times 10^{-08	}$\\ \hline
64	&	11	&$	5.05\times 10^{-07	}$\\ \hline
128	&	12	&$	1.41\times 10^{-07	}$\\ \hline
256	&	11	&$	1.42\times 10^{-07	}$\\ \hline
512	&	12	&$	2.73\times 10^{-07	}$\\ \hline
1024	&	12	&$	5.34\times 10^{-08	}$\\ \hline
2048	&	13	&$	5.18\times 10^{-06	}$\\ \hline
4096	&	13	&$	1.62\times 10^{-04	}$\\ \hline

\end{tabular}
\end{center}
\end{table}

\begin{table}[h]
\caption{Evaluation of Complex Gaussian Polynomial in the Unit Disk}
\label{PolyEvalComplex2}
\begin{center}
\begin{tabular}{|c|c|c|}
\hline
Degree	&	Max. Rank	&	Error	\\ \hline
32	&	18	&$	1.77\times 10^{-06	}$\\ \hline
64	&	13	&$	1.39\times 10^{-06	}$\\ \hline
128	&	13	&$	1.16\times 10^{-06	}$\\ \hline
256	&	12	&$	8.71\times 10^{-07	}$\\ \hline
512	&	12	&$	6.97\times 10^{-07	}$\\ \hline
1024	&	12	&$	5.40\times 10^{-07	}$\\ \hline
2048	&	13	&$	4.73\times 10^{-07	}$\\ \hline
4096	&	13	&$	3.86\times 10^{-07	}$\\ \hline
\end{tabular}
\end{center}
\end{table}

\begin{table}[h]
\caption{Polynomial Evaluation by Using the Algorithm of \cite{MB72}} 
$~~~~~~~~~~~~~~~~~~~~~~~~~~~~~~~~~~~~~~~~~~~~~$(the  entry ``Inf" means ``beyond the range")
\label{SuperfastEval}
\begin{center}
\begin{tabular}{|*{6}{c|}}
\hline
 	& \multicolumn{2}{|c|}{Real Gaussian } & \multicolumn{2}{c|}{Complex Gaussian} \\ \hline
Degree 	&	 mean 	&	 std 	&	 mean 	&	 std 	\\ \hline
16	&$	5.19\times 10^{-09	}$&$	1.21\times 10^{-08	}$&$	8.91\times 10^{-11	}$&$	6.50\times 10^{-11	}$\\ \hline
32	&$	4.54\times 10^{-02	}$&$	6.72\times 10^{-02	}$&$	1.66\times 10^{-03	}$&$	8.86\times 10^{-04	}$\\ \hline
64	&$	9.47\times 10^{+21	}$&$	2.99\times 10^{+22	}$&$	2.96\times 10^{+11	}$&$	1.22\times 10^{+11	}$\\ \hline
128	&$	2.87\times 10^{+53	}$&$	7.21\times 10^{+53	}$&$	2.12\times 10^{+164	}$&$	   Inf	$\\ \hline

\end{tabular}
\end{center}
\end{table} 

\clearpage


\section{Conclusions}\label{scnc}

  
The papers  \cite{MRT05}, \cite{CGS07},
\cite{XXG12}, and \cite{XXCB14}
combine the FMM/HSS techniques with the transformation of matrix structures
(traced back to \cite{P90}) in order to
devise
fast
algorithms that approximate the solution of
Toeplitz, Hankel, Toep\-litz-like, and Han\-kel-like linear systems
of equations by using nearly linear number of arithmetic 
operations  performed with bounded precision.
We   
yielded similar results (that is, 
used nearly linear number of arithmetic 
operations  performed with bounded precision)
for 
multiplication of Van\-der\-monde and Cauchy matrices
by a vector, the solution of linear systems of equations with these matrices,
and polynomial 
multipoint
evaluation and interpolation.
This can be compared with
quadratic  arithmetic time of the known algorithms.
The more involved techniques of 2D FMM should help to decrease our upper bounds $\alpha(M)$
by a logarithmic factor
 (cf. \cite[Section 3.6]{B10}).

Our Section \ref{scmplcv1} and the papers \cite{P15} and \cite{Pa}  cover some extensions
of our techniques and results to computations with other structured matrices
and rational functions.
Our study also promises a natural extension
 to the important class of polynomial  Van\-der\-monde matrices,
$V_{\bf P,s}=(p_j(x_i))_{i,j=0}^{m-1,n-1}$,
where ${\bf P}=(p_j(x))_{j=0}^{n-1}$
is any basis in the space of polynomials of degree less than $n$.
This extension should exploit the following
generalization of our equation (\ref{eqfhr}), 
which 
reproduces   \cite[equation (3.6.8)]{P01},

$$C_{\bf s,t}=\diag(l(s_i)^{-1})_{i=0}^{m-1}V_{\bf P,s}V^{-1}_{\bf P,t}\diag(l'(t_j))_{j=0}^{n-1},~l(x)=\prod_{j=0}^{n-1}(x-t_j).$$

For a natural  further direction, we plan to recast our algorithms
into the form of algorithms for computations with H and H2 matrices.
This will enable us to apply the efficient subroutines
 available in the HLib library developed at the
Max Planck Institute for Mathematics in the Sciences by L. Grasedyck and
S. B\"{o}rm, www.hlib.org, and in the H2Lib, http://www.h2lib.org/, 
https://github.com/H2Lib/H2Lib.



{\bf Acknowledgements:}
Our research has been supported by the NSF Grants CCF 1116736 and CCF-1563942 and
PSC CUNY Awards 67699-00 45 and 
68862--00 46.
We also greatly appreciate reviewers'  thoughtful and helpful comments.




\end{document}